\documentclass[11pt]{amsproc}
\usepackage{amssymb}
\usepackage{amsmath}
\usepackage{amsfonts}

\theoremstyle{plain}

\newtheorem{claim}{Claim}

\newtheorem{lemma}{Lemma}

\newtheorem{proposition}{Proposition}
\newtheorem{remark}{Remark}

\newtheorem{theorem}{Theorem}
\numberwithin{equation}{section}
\input{tcilatex}

\begin{document}
\title[Existence and multiplicity of solutions...]{Existence and
multiplicity of solutions to elliptic equations of fourth order on compact
manifolds.}
\author{Mohammed Benalili}
\address{University Abou-Bekr Belka\"{\i}d, Faculty of Sciences Dept. Maths
B.P.119 Tlemcen Algeria.}
\email{m\_benalili@mail.univ-tlemcen.dz}

\begin{abstract}
This paper deals with a fourth order elliptic equation on compact Riemannian
manifolds, the function $f$ involved in the nonlinearity is of changing sign
which makes the analysis more difficult than the case where $f$ \ is of
constant sign.We prove the multiplicity of solutions in the subcritical case
which is the subject of the first theorem. In the second one we establish
the existence of solutions to the equation with critical Sobolev growth.
\end{abstract}

\maketitle

\section{\protect\bigskip Introduction\qquad}

Let$\ (M,g)$ be a Riemannian compact smooth $n$- manifold $n\geq 5$ with the
metric $g$, we let $H_{2}^{2}(M)$ be the standard Sobolev space which is the
completion of the space 
\begin{equation*}
C_{2}^{2}(M)=\left\{ u\in C^{\infty }(M)\text{: }\left\Vert u\right\Vert
_{2,2}<+\infty \right\}
\end{equation*}%
with respect to the norm $\left\Vert u\right\Vert
_{2,2}=\sum_{l=0}^{2}\left\Vert \nabla ^{l}u\right\Vert _{2}$.

Let $H_{2}$ be the space $H_{2}^{2}$ endowed with the equivalent norm 
\begin{equation*}
\left\Vert u\right\Vert _{H_{2}}=\left( \left\Vert \Delta u\right\Vert
_{2}^{2}+\left\Vert \nabla u\right\Vert _{2}^{2}+\left\Vert u\right\Vert
_{2}^{2}\right) ^{\frac{1}{2}}\text{.}
\end{equation*}%
where, $\Delta (u)=-div(\nabla u)$, denotes the Riemannian Laplacian.

First we establish the existence of at least two solutions of the
subcritical equation

\begin{equation}
\Delta ^{2}u+\nabla ^{i}(a(x)\nabla _{i}u)+h(x)u=f(x)\left\vert u\right\vert
^{q-2}u  \label{1}
\end{equation}%
where $2<q<N.$ Next we investigate solutions of the critical equation 
\begin{equation}
\Delta ^{2}u+\nabla ^{i}(a(x)\nabla _{i}u)+h(x)u=f(x)\left\vert u\right\vert
^{N-2}u  \label{2}
\end{equation}%
where $a$, $h$ and $f$ \ are smooth functions on $\ M$ \ and $N=\frac{2n}{n-4%
}$ is the critical exponent.

The function $f$ involved in the nonlinearity is of changing sign which
makes the analysis more difficult than the case where $f$ \ is of constant
sign.

The equation (\ref{1}) has a geometric roots, in fact while the conformal
Laplacian%
\begin{equation*}
L_{g}(u)=\Delta u+\frac{n-2}{4(n-1)}Ru
\end{equation*}%
where $R$ is the scalar curvature of the metric $g$, is associated to the
scalar curvature; the Paneitz operator as discovered by Paneitz $\left( 
\text{\cite{10}}\right) $ on 4-dimension manifolds and extended by Branson $%
\left( \text{\cite{3}}\right) $ to higher dimensions ( $n\geq 5$ ) \ reads
as 
\begin{equation*}
PB_{g}(u)=\Delta ^{2}u+div(-\frac{(n-2)^{2}+4}{2(n-1)(n-2)}R.g+\frac{4}{n-2}%
Ric)du+\frac{n-4}{2}Q^{n}u
\end{equation*}%
where $Ric$ is the Ricci curvature of $g$ and where%
\begin{equation*}
Q^{n}=\frac{1}{2(n-1)}\Delta R+\frac{n^{3}-4n^{2}+16n-16}{8(n-1)^{2}(n-2)^{2}%
}R^{2}-\frac{2}{(n-2)^{2}}\left\vert Ric\right\vert ^{2}
\end{equation*}%
is associated to the notion of $Q$ -curvature, good references on the
subject are Chang (\cite{5}) and Chang-Yang (\cite{6}). When the manifold ($%
M,g$) is Einstein, the Paneitz-Branson operator has constant coefficients.
It expresses as 
\begin{equation*}
PB_{g}=\Delta ^{2}u+\alpha \Delta u+au
\end{equation*}%
with 
\begin{equation*}
\alpha =\frac{n^{2}-2n-4}{2n(n-1)}R\text{ \ \ and \ }a=\frac{(n-4)(n^{2}-4)}{%
16n(n-1)^{2}}R^{2}
\end{equation*}%
and this operator is a special case of what it is usually referred as a
Paneitz- Branson type operator with constant coefficients.

Since 1990 many results have been established for precise functions $a$, $h$
and $f.$ D.E. Edmunds, D. Fortunato, E. Jannelli $\left( \text{\cite{8}}%
\right) $ proved for $n\geq 8$ that if $\lambda \in (0,\lambda _{1})$, with $%
\lambda _{1}$ is the first eigenvalue of $\Delta ^{2}$ on the euclidean open
ball $B$, the problem 
\begin{equation*}
\left\{ 
\begin{array}{c}
\Delta ^{2}u-\lambda u=u\left\vert u\right\vert ^{\frac{8}{n-4}}\text{ in }B
\\ 
u=\frac{\partial u}{\partial n}=0\text{ on }\partial B%
\end{array}%
\right.
\end{equation*}%
has a non trivial solution.

In 1995, R$.$ Van der Vorst $\left( \text{\cite{12}}\right) $ obtained the
same results as D.E. Edmunds, D. Fortunato, E. Jannelli. when applied to the
problem%
\begin{equation*}
\left\{ 
\begin{array}{c}
\Delta ^{2}u-\lambda u=u\left\vert u\right\vert ^{\frac{8}{n-4}}\text{ in }%
\Omega \\ 
u=\Delta u=0\text{ \ on \ }\partial \Omega%
\end{array}%
\right.
\end{equation*}%
where $\Omega $ is an open bounded set of $R^{n}$ and moreover he showed
that the solution is positive

In (\cite{7}) D.Caraffa studied the equation $\left( \text{\ref{1}}\right) $
on compact manifolds in the case $f(x)=$constant; and in the particular case
where the functions $a(x)$ and $h(x)$ are precise constants she obtained the
existence of positive regular solutions.

In the case of second order equation related to the prescribed scalar
curvature, that is 
\begin{equation}
\Delta u+\frac{n-2}{4(n-1)}Ru=fu^{2^{\ast }-1}  \label{3}
\end{equation}%
where $2^{\ast }=\frac{2n}{n-2}$, A. Rauzy \cite{11} stated, in the case
where the scalar curvature $R$ of the manifold $(M,g)$ is a negative
constant and $f$ \ is a changing sign function, the following results.

Let $f$ \ be a $C^{\infty }$ function on $\ M$, $f^{-}=-\inf (f,0)$, $%
f^{+}=\sup (f,0)$ and 
\begin{equation*}
\lambda _{f}=\underset{u\in A}{\inf }\frac{\int_{M}\left\vert \nabla
u\right\vert ^{2}dv_{g}}{\int_{M}u^{2}dvg}
\end{equation*}%
where $A=\left\{ u\in H_{1}^{2}(M),u\geq 0,\ \ u\not\equiv 0\text{ s.t. }%
\int_{M}f^{-}udv_{g}=0\right\} $, and $\lambda _{f}=+\infty $ \ if $A=\phi $.

\begin{theorem}
\label{1'} Let $\left( M,g\right) $\ be a smooth manifold with constant
scalar curvature $R<0$\ and let $f$\ be a smooth changing sign function on $%
M $. Suppose that there exists a constant $C>0$\ which depends only on $%
\frac{f^{-}}{\int_{M}f^{-}dv_{g}}$\ such that if $f$\ fulfills the following
conditions 
\begin{equation*}
\left. 
\begin{array}{l}
(1)\text{ }\left\vert R\right\vert <\frac{4(n-1)}{n-2}\lambda _{f}\text{ \ \
\ \ \ \ \ \ \ \ \ \ \ \ \ \ \ \ \ \ \ \ \ \ \ \ \ \ \ \ \ \ \ \ \ \ \ \ \ \
\ \ \ \ \ \ \ \ \ \ \ \ \ \ \ \ \ \ \ \ \ \ \ \ \ \ \ \ \ \ \ \ \ \ \ \ \ \
\ \ \ \ \ \ \ \ \ \ \ \ \ \ \ \ \ \ \ \ } \\ 
(2)\text{ }\frac{\sup f^{+}}{\int f^{-}dv_{g}}<C.%
\end{array}%
\right.
\end{equation*}%
Then, the equation (\ref{3})\ admits a positive solution.
\end{theorem}

\begin{theorem}
\label{2'} Let $\left( M,g\right) $\ be a smooth manifold with constant
scalar curvature $R<0$\ and let $f$\ be a smooth changing sign function on $%
M $. Suppose that there exists a constant $C>0$\ which depends only on $%
\frac{f^{-}}{\int_{M}f^{-}dv_{g}}$\ such that if $f$\ fulfills the following
conditions 
\begin{equation*}
\left. 
\begin{array}{l}
(1)\text{ }\left\vert R\right\vert <\frac{4(n-1)}{n-2}\lambda _{f}\text{ \ \
\ \ \ \ \ \ \ \ \ \ \ \ \ \ \ \ \ \ \ \ \ \ \ \ \ \ \ \ \ \ \ \ \ \ \ \ \ \
\ \ \ \ \ \ \ \ \ \ \ \ \ \ \ \ \ \ \ \ \ \ \ \ \ \ \ \ \ \ \ \ \ \ \ \ \ \
\ \ \ \ \ \ \ \ \ \ \ \ \ \ \ \ \ \ \ \ } \\ 
(2)\text{ }\frac{\sup f^{+}}{\int f^{-}dv_{g}}<C \\ 
(3)\text{ }\sup_{M}f>0.%
\end{array}%
\right.
\end{equation*}%
\textit{Then the subcritical equation }$\Delta _{g}u+Ru=fu^{q-1},$\textit{\ }%
$q\in \left] 2,2^{\ast }\right[ $\textit{\ admits two nontrivial distinct
solutions.}
\end{theorem}

More recently \cite{2} the authors have extended the work of Rauzy to the
case of the so called generalized prescribed scalar curvature type equation

\begin{equation}
\Delta _{p}u+au^{p-1}=fu^{p^{\ast }-1}  \label{4}
\end{equation}%
where $p^{\ast }=\frac{np}{n-p}$, $\Delta _{p}u=-div(\left\vert \nabla
u\right\vert ^{p-2}\nabla u)$ is the $p$-Laplacian operator on a compact
manifold $M$ of dimension $n\geq 3$, with negative scalar curvature, $p\in
(1,n)$, $u\in H_{1}^{p}(M)$ is a positive function, $f$ \ is a changing sign
function and $a$ is a negative constant. Let%
\begin{equation*}
\lambda _{f}=\underset{u\in A}{\inf }\frac{\int_{M}\left\vert \nabla
u\right\vert ^{p}dv_{g}}{\int_{M}u^{p}dvg}
\end{equation*}%
where $A=\left\{ u\in H_{1}^{p}(M),u\geq 0\text{, }u\not\equiv 0\text{ s.t. }%
\int_{M}f^{-}udv_{g}=0\right\} $, and $\lambda _{f}=+\infty $ \ if $A=\phi $

\begin{theorem}
\label{3'}(Critical case) \ There is a constant $C>0$ which depends only on $%
f^{-}/(\int f^{-}dv_{g})$ such that if $f\in C^{\infty }$ on $M$ fulfills
the following conditions

(1) $\left\vert a\right\vert $ $<$ $\lambda _{f}$

(2) $\left( supf^{+}/\int f^{-}dv_{g}\right) <C$.

Then the equation $\left( \text{\ref{4}}\right) $ \ has a positive solution
of class $C^{1,\alpha }(M)$.
\end{theorem}

\begin{theorem}
\label{4'} (Subcritical case ) \ For every $C^{\infty }$- function on $M$
there is a constant $C>0$ which depends only on $f^{-}/(\int f^{-}dv_{g})$
such that if $f$ fulfills the following conditions

(1) $\left\vert a\right\vert $ $<$ $\lambda _{f}$

(2) $\left( supf^{+}/\int f^{-}dv_{g}\right) <C$

(3) $\sup f>0.$

Then the subcritical equation 
\begin{equation*}
\Delta _{p}u+au^{p-1}=fu^{q-1}\text{ \ }q\in \left] p,p^{\ast }\right[
\end{equation*}
\ has at least two non trivial positive solutions of class $C^{1,\alpha }(M) 
$.
\end{theorem}

For $a$,$\ f$ , $C^{\infty }$ -functions $M$, we let $\ \ $%
\begin{equation*}
\lambda _{a,f}=\underset{u\in A}{\inf }\frac{\int_{M}(\Delta
u)^{2}dv_{g}-\int_{M}a\left\vert \nabla u\right\vert ^{2}dv_{g}}{%
\int_{M}u^{2}dvg}
\end{equation*}%
where $A=\left\{ u\in H_{2}\text{, }u\geq 0\text{, }u\not\equiv 0\text{ s.
t. }\int_{M}f^{-}udv_{g}=0\right\} ,~\ $and%
\begin{equation*}
\lambda _{a,f}=+\infty \ \ \ \ \text{if\ \ \ \ \ \ \ \ \ }A=\phi \text{.}
\end{equation*}%
Let $h$ be a smooth negative function on $M$, we consider the functional $%
F_{q}$ defined on $H_{2}$ by 
\begin{equation*}
F_{q}(u)=\left\Vert \Delta u\right\Vert _{2}^{2}-\int_{M}a\left\vert \nabla
u\right\vert ^{2}dv_{g}+\int_{M}hu^{2}dv_{g}-\int_{M}f\left\vert
u\right\vert ^{q}dv_{g},\text{ \ \ \ \ }q\in \left( 2,N\right] \text{.}
\end{equation*}%
In the case of fourth order elliptic equations on manifolds with changing
sign right hand side, no work is done at least I know off. While we borrow
ideas from the paper of Rauzy (\cite{12}), our method is not an adaptation
of that of Rauzy, since the behavior of fourth order operators differs from
that of second order ones. It is essentially due to the structures of the
spaces $H_{1}^{2}\left( M\right) $ and $H_{2}^{2}(M)$: indeed if $u\in
H_{1}^{2}\left( M\right) $ so does $\left\vert u\right\vert $ and the
gradient of $\left\vert u\right\vert $ satisfies $\left\vert \nabla
\left\vert u\right\vert \right\vert =\left\vert \nabla u\right\vert $ and
also the analysis on $H_{2}^{2}\left( M\right) $ is more complicated than on 
$H_{1}^{2}\left( M\right) $. In this paper we state the following results

\begin{theorem}
\label{th1} Let $a$, $h$ be $C^{\infty }$ functions on $M$ with $h$
negative. For every $C^{\infty }$ function, $f$ on $M$ with $%
\int_{M}f^{-}dv_{g}>0$, there exists a constant $C>0$ which depends only on $%
\frac{f^{-}}{\int f^{-}dv_{g}}$ such that if $f$ satisfies the following
conditions 
\begin{equation*}
\left. 
\begin{array}{l}
(1)\text{ }\left\vert h(x)\right\vert <\lambda _{a,f}\text{ \ \ \ \ \ \ \ \
\ \ \ \ \ \ \ \ \ for any }x\in M\text{\ \ \ \ \ \ \ \ \ \ \ \ \ \ \ \ \ \ \
\ \ \ \ \ \ \ \ \ \ \ \ \ \ \ \ \ \ \ \ \ \ \ \ \ \ \ \ \ \ \ \ \ \ \ \ \ \
\ \ \ \ \ \ \ \ \ \ \ \ \ \ \ \ } \\ 
(2)\text{ }\frac{\sup f^{+}}{\int f^{-}dv_{g}}<C \\ 
(3)\text{ }\sup_{M}f>0,%
\end{array}%
\right.
\end{equation*}%
then the subcritical equation 
\begin{equation*}
\Delta ^{2}u+\nabla ^{i}(a\nabla _{i}u)+hu=f\left\vert u\right\vert ^{q-2}u,%
\text{ \ \ \ \ }q\in \left] 2,N\right[
\end{equation*}%
has at least two distinct solutions $u$ and $v$ satisfying $F_{q}\left(
u\right) <0<F_{q}\left( v\right) $ and of class $C^{4,\alpha }$, for some $%
\alpha \in (0,1)$.
\end{theorem}

\begin{theorem}
\label{th2} Let $a$, $h$ be $C^{\infty }$ functions on $M$ with $h$
negative.\ For every $C^{\infty }$ function $f$ on $M$ with $%
\int_{M}f^{-}dv_{g}>0$ there exists a constant $C>0$ which depends only on $%
\frac{f^{-}}{\int f^{-}dv_{g}}$ such that if $f$ satisfies the following
conditions 
\begin{equation*}
\left. 
\begin{array}{l}
(1)\text{ }\left\vert h(x)\right\vert <\lambda _{a,f}\text{ \ \ \ \ \ \ \
for any }x\in M\text{\ \ \ \ \ \ \ \ \ \ \ \ \ \ \ \ \ \ \ \ \ \ \ \ \ \ \ \
\ \ \ \ \ \ \ \ \ \ \ \ \ \ \ \ \ \ \ \ \ \ \ \ \ \ \ \ \ \ \ \ \ \ \ \ \ \
\ \ \ \ \ \ \ \ \ \ \ \ \ \ \ \ \ \ \ \ \ \ \ \ \ } \\ 
(2)\text{ }\frac{\sup f^{+}}{\int f^{-}dv_{g}}<C%
\end{array}%
\right.
\end{equation*}%
the critical equation%
\begin{equation*}
\Delta ^{2}u+\nabla ^{i}(a\nabla _{i}u)+hu=f\left\vert u\right\vert ^{N-2}u
\end{equation*}%
has a solution of class $C^{4,\alpha }$, for some $\alpha \in (0,1)$, \ with
negative energy.
\end{theorem}

To have applications to conformal geometry, we must obtain positive
solutions but this is a difficult problem because of the lack of a maximum
principle, This will be treated in a separated work.

If \ the set $A=\phi $, the condition $(1)$ of Theorem \ref{th1} and \ref%
{th2} is fulfilled.

Suppose that $A\neq \phi $ and let $\mu =\inf_{u\in A}\frac{%
\int_{M}\left\vert \nabla u\right\vert dv_{g}}{\int_{M}u^{2}dv_{g}}$.

\begin{remark}
We get smooth functions for which we have solutions by observing that $%
\int_{\left\{ x\in M\ :\ f(x)\geqslant 0\right\} }dv_{g}<\left(
K_{2}^{2}+\epsilon \right) \left\Vert h\right\Vert _{\infty }+A_{2}\left(
\epsilon \right) +\mu \left\Vert a\right\Vert _{\infty }$ implies $\lambda
_{a,f}>\left\Vert h\right\Vert _{\infty }$ ( See Lemma \ref{lem1'} ) where $%
\epsilon $ is any positive real number and $K_{2}$, $A_{2}\left( \epsilon
\right) $ are the constants of the Sobolev inequality given by Lemma \ref%
{lem1}.
\end{remark}

Let $B_{k,q}=\left\{ u\in H_{2}:\ \left\Vert u\right\Vert _{q}^{q}=k\right\}
,$ where $\left\Vert {}\right\Vert _{q}$ denotes the $L^{q}$-norm, and put $%
\mu _{k,q}=\inf_{u\in B_{k,q}}F_{q}(u)$. \ The method used in this paper
consists in the case of Theorem \ref{th1}, to show that the curve $%
k\rightarrow \mu _{k,q}$ is continuous as a function of the argument $k,$
starts at $0$ goes by a relative negative minimum, which is attained, and
takes positive values for $k$ in some interval $l_{q}$ and finally goes to $%
-\infty $, to do so many a priori estimates are given, then we deduce the
existence of two solutions of the subcritical equation, one of negative
energy and the other of positive energy. For the proof of Thorem \ref{th2},
we show that the sequence of solutions of the subcritical equations, with
negative energies,obtained in Theorem \ref{th1} is bounded in $H_{2}$ as $q$
tends to $N=\frac{2n}{n-4}$, the critical Sobolev exponent. By classical
arguments, we show that up to a subsequence $u_{q}$ converges weakly to a
solution $u$ of the critical equation. After, we show that $u$ is of
negative energy i.e. $u\not\equiv 0$.

\section{Preliminaries}

Let $a$, $h$ be $C^{\infty }$ functions on $M$ with $h$ negative. We suppose
without lost of generality that the Riemannian manifold $(M,g)$ is of volume
equals to $1$. Since it is equivalent to solve the equation (\ref{1}) with $%
f $ \ or $\alpha f$ $\ $( $\alpha $ a real number $\neq 0$ ), we consider
the functional $F_{q}$ defined on $H_{2}$ by 
\begin{equation*}
F_{q}(u)=\left\Vert \Delta u\right\Vert _{2}^{2}-\int_{M}a\left\vert \nabla
u\right\vert ^{2}dv_{g}+\int_{M}hu^{2}dv_{g}-\int_{M}f\left\vert
u\right\vert ^{q}dv_{g},\text{ \ \ \ \ }q\in \left( 2,N\right)
\end{equation*}%
and set 
\begin{equation*}
B_{k,q}=\left\{ u\in H_{2}(M),\text{ }\left\Vert u\right\Vert
_{q}^{q}=k\right\}
\end{equation*}%
where $k$ is some constant. Let%
\begin{equation*}
\mu _{k,q}=\underset{u\in B_{k,q}}{\inf }F_{q}(u)\text{,}
\end{equation*}%
we state

\begin{proposition}
\label{prop1} The infimum $\ \mu _{k,q}$ is achieved. Futhermore any
minimizer of the functional $F_{q}$ is of class $C^{4,\alpha }$, $\alpha \in
\left( 0,1\right) $.
\end{proposition}

\begin{proof}
We have 
\begin{equation}
F_{q}(u)\geq \left\Vert \Delta u\right\Vert _{2}^{2}-\left\Vert
a_{+}\right\Vert _{\infty }\left\Vert \nabla u\right\Vert _{2}^{2}+k^{\frac{2%
}{q}}\min_{x\in M}h(x)  \label{5}
\end{equation}%
\begin{equation*}
-k\max_{x\in M}f(x).
\end{equation*}%
where $a_{+}(x)=\max \left[ a(x),0\right] $ and $\left\Vert .\right\Vert
_{\infty }$ is the supremum norm.

The following formula is well known on compact manifolds%
\begin{equation*}
\left\Vert \nabla ^{2}u\right\Vert _{2}^{2}\leq \left\Vert \Delta
u\right\Vert _{2}^{2}-\int_{M}Ric_{ij}\nabla u_{i}\nabla u_{j}dv_{g}
\end{equation*}%
\begin{equation}
\leq \left\Vert \Delta u\right\Vert _{2}^{2}+\beta \left\Vert \nabla
u\right\Vert _{2}^{2}.  \label{6}
\end{equation}%
where $\beta $ is some constant. As it is shown in (\cite{1} p.93), for any $%
\eta >0$, there exists a constant $C(\eta )$ depending on $\eta $ such that 
\begin{equation}
\left\Vert \nabla u\right\Vert _{2}^{2}\leq \eta \left\Vert \nabla
^{2}u\right\Vert _{2}^{2}+C(\eta )\left\Vert u\right\Vert _{2}^{2}  \label{7}
\end{equation}%
Plugging (\ref{6}) in (\ref{7}), we get%
\begin{equation}
\left\Vert \nabla u\right\Vert _{2}^{2}\leq \eta \left\Vert \Delta
u\right\Vert _{2}^{2}+\eta \beta \left\Vert \nabla u\right\Vert
_{2}^{2}+C(\eta )\left\Vert u\right\Vert _{2}^{2}  \label{8}
\end{equation}%
and choosing $\eta $ such that $\eta \beta \leq \frac{1}{2}$, we obtain%
\begin{equation}
\left\Vert \nabla u\right\Vert _{2}^{2}\leq 2\eta \left\Vert \Delta
u\right\Vert _{2}^{2}+2C(\eta )\left\Vert u\right\Vert _{2}^{2}.  \label{9}
\end{equation}%
The inequality (\ref{5}) reads then 
\begin{equation*}
F_{q}(u)\geq \left\Vert \Delta u\right\Vert _{2}^{2}\left( 1-2\eta
\left\Vert a_{+}\right\Vert _{\infty }\right)
\end{equation*}%
\begin{equation*}
+k^{\frac{2}{q}}\left( \min_{x\in M}h(x)-2C(\eta )\left\Vert
a_{+}\right\Vert _{\infty }\right) -k\max_{x\in M}f(x)
\end{equation*}%
and then, with $\eta $ small enough, we have \ 
\begin{equation*}
1-2\eta \left\Vert a_{+}\right\Vert _{\infty }=\alpha >0
\end{equation*}%
so%
\begin{equation}
F_{q}(u)\geq \alpha \left\Vert \Delta u\right\Vert _{2}^{2}+C_{1}  \label{10}
\end{equation}%
where $\alpha $ is some positive constant and $C_{_{1}\text{ }}$is a
constant independent of $u$. Let $(u_{j})$ be a minimizing sequence of the
functional $F_{q}$ in $B_{k,q}$; \ so for $j$ sufficiently large $%
F_{q}(u_{j})\leq $ $\mu _{k,q}+1$ and by (\ref{10}), we get 
\begin{equation*}
\left\Vert \Delta u_{j}\right\Vert _{2}^{2}\leq \frac{1}{\alpha }\left( \mu
_{k,q}+1-C_{1}\right) \text{.}
\end{equation*}%
By formula (\ref{9}) and the fact 
\begin{equation*}
\left\Vert u_{j}\right\Vert _{2}^{2}\leq k^{\frac{2}{q}}\text{,}
\end{equation*}%
we obtain that $\left\Vert \nabla u_{j}\right\Vert _{2}^{2}$ is bounded. It
follows that the sequence $(u_{j})$ is bounded in $H_{2}.$ Consequently $%
u_{j}$ converges weakly in $H_{2}$, the compact embedding of $H_{2}$ in $%
L^{q}$ and the unicity of the weak limit allow us to claim that there is a
subsequence of $(u_{j})$ still denoted $\left( u_{j}\right) $ such that 
\begin{equation*}
u_{j}\rightarrow u\text{ strongly in }L^{s}\text{ \ \ for any }s<N
\end{equation*}%
\begin{equation*}
\nabla u_{j}\rightarrow \nabla u\text{ strongly in }L^{2}
\end{equation*}%
and%
\begin{equation*}
\left\Vert u\right\Vert _{H_{2}}\leq \lim_{j}\inf \left\Vert
u_{j}\right\Vert _{H_{2}}\text{ .}
\end{equation*}%
Consequently%
\begin{equation*}
F_{q}(u)=\left\Vert \Delta u\right\Vert _{2}^{2}-\int_{M}a\left\vert \nabla
u\right\vert ^{2}dv_{g}+\int_{M}hu^{2}dv_{g}-\int_{M}f\left\vert
u\right\vert ^{q}dv_{g}
\end{equation*}%
\begin{equation*}
\leq \lim \inf_{j}\left\Vert \Delta u_{j}\right\Vert
_{2}^{2}-\lim_{j}\int_{M}a\left\vert \nabla u_{j}\right\vert
^{2}dv_{g}+\lim_{j}\int_{M}hu_{j}^{2}dv_{g}-\lim_{j}\int_{M}f\left\vert
u_{J}\right\vert ^{q}dv_{g}
\end{equation*}%
\begin{equation*}
=\lim_{J}F_{q}(u_{j})=\mu _{k,q}
\end{equation*}%
and since clearly%
\begin{equation*}
\left\Vert u\right\Vert _{q}^{q}=k
\end{equation*}%
we obtain that 
\begin{equation*}
F_{q}(u)=\mu _{k,q}\text{ .\ }
\end{equation*}%
So $u$ fulfills 
\begin{equation*}
\int_{M}\Delta u.\Delta vdv_{g}-\int_{M}a(x)\nabla ^{i}u.\nabla
_{i}vdv_{g}+\int_{M}h(x)uvdv_{g}
\end{equation*}%
\begin{equation*}
-\frac{q}{2}\int_{M}f(x)\left\vert u\right\vert ^{q-2}uvdv_{g}=\lambda
_{k,q}\int_{M}\left\vert u\right\vert ^{q-2}uvdv_{g}
\end{equation*}%
for any $v\in H_{2}$; where $\lambda _{k,q}$ is the Lagrange multiplier and $%
u$ is a weak solution of the equation%
\begin{equation}
\Delta ^{2}u+\nabla ^{i}(a\nabla _{i}u)+hu=\left( \lambda _{k,q}+\frac{q}{2}%
f\right) \left\vert u\right\vert ^{q-2}u\text{.}  \label{11}
\end{equation}

Using the bootstrap method, we show that $u\in L^{s}(M)$ \ for any $s<N$, so 
$P(u)=\Delta ^{2}u+\nabla ^{i}(a\nabla _{i}u)+hu\in L^{s}(M)$ \ for any $s<N$
and since $P$ is a fourth order elliptic operator, it follows by a well
known regularity theorem that $P(u)\in C^{0,\alpha }(M)$ for some $\alpha
\in (0,1)$. Then $u\in C^{4,\alpha }(M)$ .
\end{proof}

\begin{proposition}
\label{prop2} $\ \ \mu _{k,q}$ is continuous as a function of the argument $%
k $ .
\end{proposition}

\begin{proof}
For any $k$ , $l\in R^{+}$, let $u$ and $v$ be two functions of norm $1$ in $%
L^{q}$ such that $F_{q}(k^{\frac{1}{q}}u)=\mu _{k,q}$ \ and $F_{q}(l^{\frac{1%
}{q}}v)=\mu _{l,q}$ .

Then 
\begin{equation*}
\mu _{l,q}-\mu _{k,q}=F_{q}(l^{\frac{1}{q}}v)-F_{q}(k^{\frac{1}{q}%
}v)+F_{q}(k^{\frac{1}{q}}v)-\mu _{k,q}
\end{equation*}%
\begin{equation*}
=F_{q}(k^{\frac{1}{q}}v)-\mu _{k,q}
\end{equation*}%
\begin{equation*}
+(l^{\frac{2}{q}}-k^{\frac{2}{q}})\left( \left\Vert \Delta v\right\Vert
_{2}^{2}-\int_{M}a\left\vert \nabla v\right\vert
^{2}dv_{g}+\int_{M}hv^{2}dv_{g}\right)
\end{equation*}%
\begin{equation*}
-(l-k)\int_{M}f\left\vert v\right\vert ^{q}dv_{g}\text{.}
\end{equation*}%
On the other hand, we have%
\begin{equation*}
\mu _{l,q}=F_{q}(l^{\frac{1}{q}}v)=l^{\frac{2}{q}}\left( \left\Vert \Delta
v\right\Vert _{2}^{2}-\int_{M}a\left\vert \nabla v\right\vert
^{2}dv_{g}+\int_{M}hv^{2}dv_{g}\right) -l\int_{M}f\left\vert v\right\vert
^{q}dv_{g}
\end{equation*}%
\begin{equation*}
\leq F_{q}(l^{\frac{1}{q}})=l^{\frac{2}{q}}\int_{M}hdv_{g}-l\int_{M}fdv_{g}
\end{equation*}%
i.e.%
\begin{equation*}
\left\Vert \Delta v\right\Vert _{2}^{2}-\int_{M}a\left\vert \nabla
v\right\vert ^{2}dv_{g}+\int_{M}hv^{2}dv_{g}\leq
\end{equation*}%
\begin{equation*}
\int_{M}hdv_{g}-l^{1-\frac{2}{q}}\int_{M}fdv_{g}+l^{1-\frac{2}{q}%
}\int_{M}f\left\vert v\right\vert ^{q}dv_{g}\text{.}
\end{equation*}%
Since $\left\Vert v\right\Vert _{q}^{q}=1$, it follows that the term $%
\int_{M}f\left\vert v\right\vert ^{q}dv_{g}$ is bounded for any $l$ in a
neighborhood of $k$ and so the term $\left\Vert \Delta v\right\Vert
_{2}^{2}-\int_{M}a\left\vert \nabla v\right\vert
^{2}dv_{g}+\int_{M}hv^{2}dv_{g}$ is upper bounded. Also since $\mu _{l,q}$
is lower bounded, it follows that $\left\Vert \Delta v\right\Vert
_{2}^{2}-\int_{M}a\left\vert \nabla v\right\vert
^{2}dv_{g}+\int_{M}hv^{2}dv_{g}$ \ is bounded in a neighborhood of $k$.

Consequently 
\begin{equation*}
\lim_{l\rightarrow k}\inf (\mu _{l,q}-\mu _{k,q})\geq \lim_{l\rightarrow
k}\inf \left( F_{q}(k^{\frac{1}{q}}v)-\mu _{k,q}\right)
\end{equation*}%
and by the definition of $\mu _{k,q}$, we get 
\begin{equation}
\lim_{l\rightarrow k}\inf (\mu _{l,q}-\mu _{k,q})\geq 0\text{ .}  \label{12}
\end{equation}%
By writing%
\begin{equation*}
\mu _{l,q}-\mu _{k,q}=\mu _{l,q}-F_{q}(l^{\frac{1}{q}}u)+F_{q}(l^{\frac{1}{q}%
}u)-F_{q}(k^{\frac{1}{q}}u)
\end{equation*}%
\begin{equation*}
=\mu _{l,q}-F_{q}(l^{\frac{1}{q}}u)
\end{equation*}%
\begin{equation*}
+(l^{\frac{2}{q}}-k^{\frac{2}{q}})\left( \left\Vert \Delta u\right\Vert
_{2}^{2}-\int_{M}a\left\vert \nabla u\right\vert
^{2}dv_{g}+\int_{M}hu^{2}dv_{g}\right)
\end{equation*}%
\begin{equation*}
-(l-k)\int_{M}f\left\vert u\right\vert ^{q}dv_{g}
\end{equation*}%
we get 
\begin{equation*}
\lim_{l\rightarrow k}\sup (\mu _{l,q}-\mu _{k,q})\leq 0\text{ }
\end{equation*}%
and taking into account of (\ref{12}), we obtain%
\begin{equation*}
\lim_{l\rightarrow k}\mu _{l,q}=\mu _{k,q}\text{ .}
\end{equation*}
\end{proof}

\section{A priori estimates}

First, we quote the following Lemma due to Djadli-Hebey-Ledoux and improved
by Hebey \cite{9}.

\begin{lemma}
\label{lem1} Let $M$ be a Riemannian compact manifold with dimension $n\geq
5 $. For any $\epsilon >0$ there is a constant $A_{2}(\epsilon )$ such that
for any $u\in H_{2}$ \ $\left\Vert u\right\Vert _{N}^{2}\leq
K_{2}^{2}(1+\epsilon )\left\Vert \Delta u\right\Vert _{2}^{2}+A_{2}(\epsilon
)\left\Vert u\right\Vert _{2}^{2}$ with $K_{2}^{-2}=\pi
^{2}n(n-4)(n^{2}-4)\Gamma \left( \frac{n}{2}\right) ^{\frac{4}{n}}\Gamma
\left( n\right) ^{-\frac{4}{n}}$.
\end{lemma}

Suppose that the set $A=\left\{ u\in H_{2},\text{ }u\not\equiv 0\text{ s. t. 
}\int_{M}f^{-}\left\vert u\right\vert dv_{g}=0\right\} $ is non empty.

\begin{lemma}
\label{lem1'} If $\int_{\left\{ x\in M\ :\ f\left( x\right) \geqslant
0\right\} }dv_{g}$ as a function of the variable $f$ tends to $0$, $\lambda
_{a,f}$ goes to $+\infty $. In particular the condition\ $\int_{\left\{ x\in
M\ :\ f(x)\geqslant 0\right\} }dv_{g}<K_{2}^{2}\left( 1+\epsilon \right)
\left\Vert h\right\Vert _{\infty }+A_{2}\left( \epsilon \right) +\mu
\left\Vert a\right\Vert _{\infty }\ \ $\ implies that $\ \ \lambda
_{a,f}>\left\Vert h\right\Vert _{\infty }$.
\end{lemma}

\begin{proof}
For any $u\in A$, we obtain by applying successively the H\"{o}lder
inequality and the Sobolev one given by Lemma \ref{lem1}, 
\begin{equation*}
\int_{\left\{ x\in M\ :\ f(x)\geqslant 0\right\} }u^{2}dv_{g}\leq \left(
\int_{\left\{ x\in M\ :\ f(x)\geqslant 0\right\} }\left\vert u\right\vert
^{N}dv_{g}\right) ^{\frac{2}{N}}\left( \int_{\left\{ x\in M\ :\
f(x)\geqslant 0\right\} }dv_{g}\right) ^{1-\frac{2}{N}}
\end{equation*}%
\begin{equation*}
=\left( \int_{M}\left\vert u\right\vert ^{N}dv_{g}\right) ^{\frac{2}{N}%
}\left( \int_{\left\{ x\in M\ :\ f(x)\geqslant 0\right\} }dv_{g}\right) ^{%
\frac{4}{n}}
\end{equation*}%
\begin{equation*}
\leq \left( K_{2}^{2}\left( 1+\epsilon \right) \left\Vert \Delta
u\right\Vert _{2}^{2}+A_{2}\left( \epsilon \right) \left\Vert u\right\Vert
_{2}^{2}\right) \left( \int_{\left\{ x\in M\ :\ f(x)\geqslant 0\right\}
}dv_{g}\right) ^{\frac{4}{n}}\text{.}
\end{equation*}%
So 
\begin{equation*}
\left( \int_{\left\{ x\in M\ :\ f(x)\geqslant 0\right\} }dv_{g}\right) ^{-%
\frac{4}{n}}\leq K_{2}^{2}\left( 1+\epsilon \right) \lambda
_{a,f}+A_{2}\left( \epsilon \right) +\inf_{x\in A}\frac{\int_{M}a\left(
x\right) \left\vert \nabla u\right\vert ^{2}dv_{g}}{\left\Vert u\right\Vert
_{2}^{2}}
\end{equation*}%
and letting $\mu =\inf_{x\in A}\frac{\int_{M}\left\vert \nabla u\right\vert
^{2}dv_{g}}{\left\Vert u\right\Vert _{2}^{2}}$ , we get that%
\begin{equation*}
\lambda _{a,f}\geqslant \frac{1}{K_{2}^{2}\left( 1+\epsilon \right) }\left(
\left( \int_{\left\{ x\in M\ :\ f(x)\geqslant 0\right\} }dv_{g}\right) ^{-%
\frac{4}{n}}-A_{2}\left( \epsilon \right) -\mu \left\Vert a\right\Vert
_{\infty }\right)
\end{equation*}%
where $\left\Vert a\right\Vert _{\infty }=\sup_{x\in M}\left\vert
a(x)\right\vert $.

Hence if $\int_{\left\{ x\in M\ :\ f(x)\geqslant 0\right\} }dv_{g}$ tends to 
$0$ as \ a function of the variable $f$ , $\lambda _{a,f}$ goes to $+\infty $%
.
\end{proof}

Denote also by $\left\Vert h\right\Vert _{\infty }=\sup_{x\in M}\left\vert
h(x)\right\vert $ the supremum norm.

As in \cite{11}, we define the quantities,%
\begin{equation*}
\lambda _{a,f,\eta ,q}=\inf_{u\in A\left( \eta ,q\right) }\frac{\left\Vert
\Delta u\right\Vert _{2}^{2}-\int_{M}a\left\vert \nabla u\right\vert ^{2}dvg%
}{\left\Vert u\right\Vert _{2}^{2}}
\end{equation*}%
with%
\begin{equation*}
A\left( \eta ,q\right) =\left\{ u\in H_{2}:\text{\ }\left\Vert u\right\Vert
_{q}=1,\text{ }\int_{M}f^{-}\left\vert u\right\vert ^{q}dv_{g}=\eta
\int_{M}f^{-}dv_{g}\right\}
\end{equation*}%
for a real $\eta >0,$ \ 

and 
\begin{equation*}
\lambda _{a,f,\eta ,q}^{\prime }=\inf_{u\in A^{\prime }\left( \eta ,q\right)
}\frac{\left\Vert \Delta u\right\Vert _{2}^{2}-\int_{M}a\left\vert \nabla
u\right\vert ^{2}dvg}{\left\Vert u\right\Vert _{2}^{2}}
\end{equation*}%
where 
\begin{equation*}
A^{\prime }\left( \eta ,q\right) =\left\{ u\in H_{2}:~\left\Vert
u\right\Vert _{q}^{q}=1\text{,\ }\int_{M}f^{-}\left\vert u\right\vert
^{q}dv_{g}\leq \eta \int_{M}f^{-}dv_{g}\right\} .
\end{equation*}%
Now, we will study $\lambda _{a,f,\eta ,q}$ , to do so, we distinguish ( as
it is done in \cite{11}) the case where the set $\left\{ x\in M:f(x)\geq
0\right\} $ is of positive measure with respect to Riemannian measure and
the case where the set is negligible and $\sup_{x\in M}f=0$.

\textit{Case }$f^{+}>0$.

\begin{claim}
For any real $\eta >0$, the set $A(\eta ,q)$ is not empty .
\end{claim}

Indeed, the set $A^{\prime }(\eta ,q)$ is not empty since it includes the
set of functions $u\in H_{2}$ such that $\left\Vert u\right\Vert _{q}=1$ and
with supports in the set $\left\{ x\in M:f^{-}(x)<\eta
\int_{M}f^{-}dv_{g}\right\} $. The same arguments as in \cite{11} show that $%
\lambda _{a,f,\eta ,q}^{\prime }$ is achieved by a function $v\in A^{\prime
}(\eta ,q)$ and moreover $v$ satisfies $\int_{M}f^{-}\left\vert u\right\vert
^{q}dv_{g}=\eta \int_{M}f^{-}dv_{g}$.

The following facts which are proved in \cite{11}, for the Laplacian
operator remain valid in the case of the bi-Laplacian operator: $\lambda
_{a,f,\eta ,q}^{\prime }$\textit{\ is a decreasing function with respect to }%
$\eta $\textit{, bounded by }$\lambda _{a,f}$\textit{\ \ and }$\lambda
_{a,f,\eta ,q}=\lambda _{a,f,\eta ,q}^{\prime }$\textit{, }so $\lambda
_{a,f,\eta ,q}$\ is also a decreasing function with respect to $\eta ,$\ and
bounded by $\lambda _{a,f}$ .

\begin{lemma}
\label{lem2} For any $q\in \left] 2,N\right[ ,$ $\lambda _{a,f,\eta ,q}$
goes to $\lambda _{a,f}$ whenever $\eta $ goes to zero.
\end{lemma}

\begin{proof}
$\lambda _{a,f,\eta ,q}$ is attained by a family of functions labelled $%
v_{\eta ,q}$. The functions $v_{\eta ,q}$ indexed by $\eta $ are bounded in $%
H_{2}^{2}$: since 
\begin{equation*}
\left\Vert v_{\eta ,q}\right\Vert _{2}^{2}\leq \left\Vert v_{\eta
,q}\right\Vert _{q}^{2}Vol(M)^{1-\frac{2}{q}}=1
\end{equation*}%
and%
\begin{equation*}
\left\Vert \Delta v_{\eta ,q}\right\Vert _{2}^{2}-\left\Vert
a_{+}\right\Vert _{\infty }\left\Vert \nabla v_{\eta ,q}\right\Vert
_{2}^{2}\leq \lambda _{a,f,\eta ,q}\left\Vert v_{\eta ,q}\right\Vert _{2}^{2}
\end{equation*}%
\begin{equation*}
\leq \lambda _{a,f}\left\Vert v_{\eta ,q}\right\Vert _{2}^{2}\leq \lambda
_{a,f}\text{.}
\end{equation*}%
By formula (\ref{9}), for a well chosen $\varepsilon >0$, there is a
constant $C(\varepsilon )>0$ such that 
\begin{equation*}
\left\Vert \nabla v_{\eta ,q}\right\Vert _{2}^{2}\leq 2\varepsilon
\left\Vert \Delta v_{\eta ,q}\right\Vert _{2}^{2}+2C(\varepsilon )\left\Vert
v_{\eta ,q}\right\Vert _{2}^{2}
\end{equation*}%
so 
\begin{equation*}
\left\Vert \Delta v_{q,\eta }\right\Vert _{2}^{2}\leq \lambda
_{a,f}+\left\Vert a_{+}\right\Vert _{\infty }\left\Vert \nabla
v_{n,q}\right\Vert _{2}^{2}
\end{equation*}%
\begin{equation*}
\leq \lambda _{a,f}+2\left\Vert a_{+}\right\Vert _{\infty }\left(
\varepsilon \left\Vert \Delta v_{\eta ,q}\right\Vert _{2}^{2}+C(\varepsilon
)\left\Vert v_{\eta ,q}\right\Vert _{2}^{2}\right)
\end{equation*}%
and 
\begin{equation*}
\left\Vert \Delta v_{q,\eta }\right\Vert _{2}^{2}\left( 1-2\varepsilon
\left\Vert a_{+}\right\Vert _{\infty }\right) \leq \lambda
_{a,f}+2\left\Vert a_{+}\right\Vert _{\infty }C(\varepsilon )\text{.}
\end{equation*}%
By choosing $\varepsilon >0$ small enough such that 
\begin{equation*}
1-2\varepsilon \left\Vert a_{+}\right\Vert _{\infty }>0
\end{equation*}%
we get that 
\begin{equation*}
\left\Vert \Delta v_{q,\eta }\right\Vert _{2}^{2}\leq C^{\prime }(\lambda
_{a,f},\left\Vert a_{+}\right\Vert _{\infty },\varepsilon )
\end{equation*}%
where $C^{\prime }(\lambda _{a,f},\left\Vert a_{+}\right\Vert _{\infty
},\varepsilon )$ is a constant depending of $\lambda _{a,f},\left\Vert
a_{+}\right\Vert _{\infty },\varepsilon $.%
\begin{equation*}
\left\Vert \nabla v_{q,\eta }\right\Vert _{2}^{2}\leq 2\varepsilon C(\lambda
_{a,f},\left\Vert a_{+}\right\Vert _{\infty },\varepsilon )+2C(\varepsilon
)\leq C^{\prime }(\lambda _{a,f},\left\Vert a_{+}\right\Vert _{\infty
},\varepsilon )\text{.}
\end{equation*}%
Consequently the sequence $(v_{q,\eta })_{\eta }$ is bounded in $H_{2}$ and
we have

\begin{equation*}
v_{q\eta }\longrightarrow v_{q}\text{ weakly in }H_{2}\text{.}
\end{equation*}%
\begin{equation*}
v_{q\eta }\longrightarrow v_{q}\text{ strongly in }H_{r}^{2}\text{, \ \ \ \ }%
r=0,1
\end{equation*}%
\begin{equation*}
v_{q\eta }\longrightarrow v_{q}\text{ strongly in }L^{q}
\end{equation*}%
and 
\begin{equation*}
\left\Vert \Delta v_{q}\right\Vert _{2}^{2}\leq \lim_{\eta \longrightarrow
0}\inf \left\Vert \Delta v_{q\eta }\right\Vert _{2}^{2}
\end{equation*}%
Also%
\begin{equation*}
\left\Vert v_{q}\right\Vert _{q}=1.
\end{equation*}%
On the other hand 
\begin{equation*}
\int_{M}f^{-}\left\vert v_{q\eta }\right\vert ^{q}dv_{g}=\eta
\int_{M}f^{-}dv_{g}
\end{equation*}%
so 
\begin{equation*}
\int_{M}f^{-}\left\vert v_{q}\right\vert ^{q}dv_{g}=0.
\end{equation*}%
Hence 
\begin{equation*}
v_{q}\in A
\end{equation*}%
and 
\begin{equation*}
\left\Vert v_{q}\right\Vert _{2}^{2}\lambda _{a,f}\leq \left\Vert \Delta
v_{q}\right\Vert _{2}^{2}-\int_{M}a\left\vert \nabla v_{q}\right\vert
^{2}dv_{g}
\end{equation*}%
\begin{equation*}
\leq \lim_{\eta \longrightarrow 0}\inf \left( \left\Vert \Delta v_{q\eta
}\right\Vert _{2}^{2}-\int_{M}a\left\vert \nabla v_{q\eta }\right\vert
^{2}dv_{g}\right) =\lim_{\eta \longrightarrow 0}\inf \left\Vert v_{q\eta
}\right\Vert _{2}^{2}\left( \lambda _{a,f,q,\eta }\right)
\end{equation*}%
and since by construction 
\begin{equation*}
\lambda _{a,f}\geq \lambda _{a,f,q,\eta }
\end{equation*}%
we get that 
\begin{equation*}
\lim_{\eta \longrightarrow 0}\lambda _{a,f,q,\eta }=\lambda _{a,f}\text{.}
\end{equation*}
\end{proof}

\begin{lemma}
\label{lem3} Let $\varepsilon >0$, there exists $\ \eta _{o}$ such that for
any $\ \eta <\eta _{o},$ there is $q_{\eta }$ such that $\lambda
_{a,f,q,\eta }\geq \lambda _{a,f}-\varepsilon $ for any $q>q_{\eta }.$
\end{lemma}

\begin{proof}
We proceed by contradiction. Suppose that there is a $\varepsilon _{o}>0$,
such that for any $\eta $ there exists an $\eta _{o}<\eta $ \ and for any $%
q_{\eta o}$ there is $q>q_{\eta o}$ with \ $\lambda _{a,f,q,\eta }<\lambda
_{f}-\varepsilon .$ If $v_{q\eta }$ is the function in $H_{2}$ which
achieves $\lambda _{a,f,q,\eta }$ , then 
\begin{equation*}
\lambda _{a,f,q,\eta }=\frac{\left\Vert \Delta v_{q\eta }\right\Vert
_{2}^{2}-\int_{M}a\left\vert \nabla v_{q\eta }\right\vert ^{2}dv_{g}}{%
\left\Vert v_{q\eta }\right\Vert _{2}^{2}}
\end{equation*}%
with $\left\Vert v_{q\eta }\right\Vert _{q}^{q}=1$. For a convenient $\eta $%
, we choose a sequence $q$ converging to $N$ such that 
\begin{equation*}
\left\Vert \Delta v_{q\eta }\right\Vert _{2}^{2}-\int_{M}a\left\vert \nabla
v_{q\eta }\right\vert ^{2}dv_{g}<\lambda _{a,f}-\varepsilon _{o}\text{.}
\end{equation*}%
By the same argument as in the proof of Lemma \ref{lem2}, we get that the
sequence $v_{q\eta }$ indexed by $q$ is bounded in $H_{2}$ so up to a
subsequence $v_{q\eta }$ converges weakly to $v_{\eta }$ in $H_{2}$ and
strongly in $H_{r}^{2}$ , $r=0,1$. Also we have%
\begin{equation*}
\left\Vert \Delta v_{\eta }\right\Vert _{2}^{2}\leq \lim_{q\longrightarrow
N}\inf \left\Vert \Delta v_{q\eta }\right\Vert _{2}^{2}
\end{equation*}%
and by the strong convergence \ in $H_{r}^{2}$, $r=0,1$, we get 
\begin{equation*}
\left\Vert \Delta v_{\eta }\right\Vert _{2}^{2}-\int_{M}a\left\vert \nabla
v_{\eta }\right\vert ^{2}dv_{g}<\left( \lambda _{a,f}-\varepsilon
_{o}\right) \left\Vert v_{\eta }\right\Vert _{2}^{2}\text{.}
\end{equation*}

By the Sobolev inequality given in the Lemma \ref{lem1} we have for any $%
\varepsilon _{1}>0$ there is a constant $A(\varepsilon _{1})>0$ such that 
\begin{equation*}
1=\left\Vert v_{q\eta }\right\Vert _{q}^{2}\leq \left\Vert v_{q\eta
}\right\Vert _{N}^{2}\text{ \ \ ( since the manifold }M\text{ is of volume }1%
\text{ )}
\end{equation*}%
\begin{equation*}
\leq K_{2}^{2}\left( 1+\varepsilon _{1}\right) \left\Vert \Delta v_{q\eta
}\right\Vert _{2}^{2}+A(\varepsilon _{1})\left\Vert v_{q\eta }\right\Vert
_{2}^{2}
\end{equation*}%
\begin{equation*}
\leq \left[ K_{2}^{2}\left( 1+\varepsilon _{1}\right) \lambda
_{a,f}+A(\varepsilon _{1})\right] \left\Vert v_{q\eta }\right\Vert
_{2}^{2}++(K_{2}^{2}+\varepsilon _{1})\left\Vert a_{+}\right\Vert _{\infty
}\left\Vert \nabla v_{q\eta }\right\Vert _{2}^{2}
\end{equation*}%
\begin{equation*}
\leq \left[ K_{2}^{2}\left( 1+\varepsilon _{1}\right) (1+\left\Vert
a_{+}\right\Vert _{\infty })\lambda _{a,f}+A(\varepsilon _{1})\right]
\left\Vert v_{q\eta }\right\Vert _{H_{1}^{2}}^{2}.
\end{equation*}%
Consequently%
\begin{equation*}
\left\Vert v_{\eta }\right\Vert _{2}^{2}\geq \frac{1}{\left[ K_{2}^{2}\left(
1+\varepsilon _{1}\right) (1+\left\Vert a_{+}\right\Vert _{\infty })\lambda
_{a,f}+A(\varepsilon _{1})\right] )}.
\end{equation*}%
As in \cite{11} we can show that 
\begin{equation*}
\int_{M}\left\vert v_{\eta }\right\vert ^{N}dv_{g}\leq 1\text{ \ \ and }%
\int_{M}f^{-}\left\vert v_{\eta }\right\vert ^{N}dv_{g}\leq \eta
\int_{M}f^{-}dv_{g}.
\end{equation*}%
Consider the sequence of $\eta $ such that for any $q_{\eta }$, there is a $%
q>q_{\eta }$ with 
\begin{equation*}
\lambda _{a,f,q,\eta }\leq \lambda _{a,f}-\varepsilon .
\end{equation*}%
Now tending $\eta $ to $0$, if $v_{\eta }$ is the sequence corresponding to $%
\eta $ previously considered, $v_{\eta }$ is bounded in $H_{2}$ and%
\begin{equation*}
\left\Vert v_{\eta }\right\Vert _{2}^{2}\geq \frac{1}{\left[ K_{2}^{2}\left(
1+\varepsilon _{1}\right) (1+\left\Vert a_{+}\right\Vert _{\infty })\lambda
_{a,f}+A(\varepsilon _{1})\right] )}.
\end{equation*}%
so $v_{\eta }$ converges weakly to $v\neq 0$ in $H_{2}$ and strongly to $v$
in $H_{r}^{2}$, $r=0,1$ and $v$ satisfies%
\begin{equation}
\left\Vert \Delta v\right\Vert _{2}^{2}-\int_{M}a\left\vert \nabla
v\right\vert ^{2}dv_{g}\leq (\lambda _{a,f}-\varepsilon _{o})\left\Vert
v\right\Vert _{2}^{2}\text{.}  \label{12'}
\end{equation}

On the other hand 
\begin{equation*}
0\leq \int_{M}f^{-}\left\vert v\right\vert ^{N}dv_{g}\leq \lim_{\eta
\longrightarrow 0}\inf \int_{M}f^{-}\left\vert v_{\eta }\right\vert
^{N}dv_{g}\leq \lim_{\eta \rightarrow 0}\eta \int_{M}f^{-}dv_{g}=0
\end{equation*}%
then $\int_{M}f^{-}\left\vert v\right\vert dv_{g}=0$ and $v$ belongs to the
domain $A$ of definition of $\lambda _{a,f}$. Hence 
\begin{equation*}
\lambda _{a,f}\leq \frac{\left\Vert \Delta v\right\Vert
_{2}^{2}-\int_{M}a\left\vert \nabla v\right\vert ^{2}dv_{g}}{%
\int_{M}\left\vert v\right\vert ^{2}dv_{g}}\text{.}
\end{equation*}%
A contradiction with the inequality (\ref{12'}) and Lemma \ref{lem3} is
proved.
\end{proof}

\textit{Case }$f^{+}=0$.

In this case $\lambda _{a,f}$ is not defined so $\lambda _{a,f}=+\infty $.
First, we give the lemma equivalent to Lemma \ref{lem2}

\begin{lemma}
\label{lem4} Let $q\in \left] 2,N\right[ $. For any positive constant $R$,
there exists $\eta _{o}$ such that for any $\ \eta <\eta _{o}$, $\ \lambda
_{a,f,\eta ,q}\geq R.$
\end{lemma}

\begin{proof}
We argue by contradiction. It is easy to show that $\lambda _{a,f,q,\eta }$
is achieved by a function $v_{q\eta }$ in $H_{2}$ with $\left\Vert v_{q,\eta
}\right\Vert _{q}=1$. Suppose that there is $\lambda _{a,f,\eta ,q}$ bounded
when $\eta $ goes to $0$. Then%
\begin{equation*}
\left\Vert \Delta v_{q,\eta }\right\Vert _{2}^{2}-\left\Vert
a_{+}\right\Vert _{\infty }\left\Vert \nabla v_{_{q,,\eta }}\right\Vert
_{2}^{2}\leq \frac{\left\Vert \Delta v_{q,\eta }\right\Vert
_{2}^{2}-\left\Vert a_{+}\right\Vert _{\infty }\left\Vert \nabla
v_{_{q,,\eta }}\right\Vert _{2}^{2}}{\left\Vert v_{_{q,,\eta }}\right\Vert
_{2}^{2}}
\end{equation*}%
\begin{equation*}
\leq \lambda _{a,f,q,\eta }<+\infty \text{.}
\end{equation*}%
and proceeding as in the proof of Lemma \ref{lem2} we get that the sequence $%
v_{q\eta }$ indexed by $\eta $ is bounded in $H_{2}$. Consequently the
sequence $v_{q\eta }$ converges weakly to $v_{q}$ in $H_{2}$ and converges
strongly to $v_{q}$ in $H_{r}^{2}$ , $r=0,1$, and strongly to $v_{q}$ in $%
L^{q}$ as $\eta $ goes to $0$. $\int_{M}f^{-}\left\vert v_{q}\right\vert
^{q}dv_{g}=0$ which implies that $v_{q}=0$ almost everywhere and $\left\Vert
v_{q}\right\Vert _{q}=1$ which are in contradiction with each other.
\end{proof}

Now we give an analogue to Lemma \ref{lem3}.

\begin{lemma}
\bigskip \label{lem5} There exists an $\eta _{o}$ such that for any $\eta
<\eta _{o}$ there is $q_{\eta }$ such that for any $q>q_{\eta }$ we have $%
\lambda _{a,f,q,\eta }>\left\Vert h\right\Vert _{\infty }$.
\end{lemma}

The proof of this lemma is similar to the previous ones so we omit it.

Let $\sigma >0,$ any sufficient small real number, with the previous
notations we obtain by using the lemmas quoted above the following

\begin{lemma}
\label{lem6} (1) Suppose that $\sup_{M}f>0$ and $\left\Vert h\right\Vert
_{\infty }<\lambda _{a,f}$. There exists $\eta $ such that $\lambda
_{a,f,q,\eta }-\left\Vert h\right\Vert _{\infty }=\varepsilon _{o}>0$.

Let $b=\frac{\left( 1-2\sigma \left\Vert a_{+}\right\Vert _{\infty }\right)
\varepsilon _{o}}{\left[ \left( \varepsilon _{o}+\left\Vert h\right\Vert
_{\infty }+2\left\Vert a_{+}\right\Vert _{\infty }C(\sigma )\right)
K_{2}^{2}\left( 1+\varepsilon \right) +\left( 1-2\sigma \left\Vert
a_{+}\right\Vert _{\infty }\right) A(\varepsilon )\right] }$

$\mu =\inf \left( b,\left\Vert h\right\Vert _{\infty }+2\left\Vert
a_{+}\right\Vert _{\infty }C(\sigma )\right) $ and suppose that

$\ \frac{\sup_{M}f}{\int_{M}f^{-}dv_{g}}<\frac{\mu \eta }{8\left( \left\Vert
h\right\Vert _{\infty }+2\left\Vert a_{+}\right\Vert _{\infty }C(\sigma
)\right) }$, where and $K_{2}^{2}$, $A(\epsilon )$ are the constants
appearing in the Sobolev inequality given by Lemma \ref{lem1}. For any $q\in %
\left] 2,N\right[ $ there exists a non empty interval $I_{q}$\ $\subset
R^{+} $ such that for every $u\in H_{2}$ with $L^{q}$-norm $\ k^{\frac{1}{q}%
} $ and $k\in $ $I_{q}=\left[ k_{1,q},k_{2\text{,}q}\right] $ we have \ $%
F_{q}(u)\geq \frac{1}{2}\mu k^{\frac{2}{q}}$.

(2) Suppose that $\sup_{M}f=0$ and $\left\Vert h\right\Vert _{\infty
}<\lambda _{a,f}$ , there exists an interval $I_{q}=\left[ k_{1,q},+\infty %
\right[ $ such that for any $k\in I_{q}$ and any $u\in H_{2}$ with $%
\left\Vert u\right\Vert _{q}^{q}=k$, we have $F_{q}(u)\geq \frac{1}{2}\mu k^{%
\frac{2}{q}}.$
\end{lemma}

\begin{proof}
\textit{Case: }$f^{+}>0$.

Let $u\in H_{2}$ such that $\left\Vert u\right\Vert _{q}^{q}=k$.

Putting 
\begin{equation*}
G_{q}(u)=\left\Vert \Delta u\right\Vert _{2}^{2}-\int_{M}a\left\vert \nabla
u\right\vert ^{2}dv_{g}+\int_{M}hu^{2}dv_{g}+\int_{M}f^{-}\left\vert
u\right\vert ^{q}dv_{g}\text{,}
\end{equation*}%
we get

\begin{equation*}
G_{q}(u)\geq \left\Vert \Delta u\right\Vert _{2}^{2}-\left\Vert
a_{+}\right\Vert _{\infty }\left\Vert \nabla u\right\Vert
_{2}^{2}-\left\Vert h\right\Vert _{\infty }\left\Vert u\right\Vert
_{2}^{2}+\int_{M}f^{-}\left\vert u\right\vert ^{q}dv_{g}
\end{equation*}%
and taking account of (\ref{9}), we obtain that for any suitable real $%
\sigma >0$, there is a constant $C(\sigma )>0$ such that%
\begin{equation*}
G_{q}(u)\geq \left( 1-2\sigma \left\Vert a_{+}\right\Vert _{\infty }\right)
\left\Vert \Delta u\right\Vert _{2}^{2}
\end{equation*}%
\begin{equation*}
-\left( \left\Vert h\right\Vert _{\infty }+2C(\sigma )\left\Vert
a_{+}\right\Vert _{\infty }\right) \left\Vert u\right\Vert
_{2}^{2}+\int_{M}f^{-}\left\vert u\right\vert ^{q}dv_{g}.
\end{equation*}%
So if 
\begin{equation*}
\int_{M}f^{-}\left\vert u\right\vert ^{q}dvg\geq \eta k\int_{M}f^{-}dv_{g}
\end{equation*}%
then%
\begin{equation*}
G_{q}(u)\geq \left( 1-2\sigma \left\Vert a_{+}\right\Vert _{\infty }\right)
\left\Vert \Delta u\right\Vert _{2}^{2}
\end{equation*}%
\begin{equation}
-\left( \left\Vert h\right\Vert _{\infty }+2C(\sigma )\left\Vert
a_{+}\right\Vert _{\infty }\right) \left\Vert u\right\Vert _{2}^{2}+\eta
k\int_{M}f^{-}dv_{g}  \label{13}
\end{equation}%
with $\sigma >0$ sufficiently small so that 
\begin{equation*}
1-2\sigma \left\Vert a_{+}\right\Vert _{\infty }>0.
\end{equation*}%
Now since%
\begin{equation*}
\left\Vert u\right\Vert _{2}^{2}\leq \left\Vert u\right\Vert _{q}^{\frac{2}{q%
}}Vol(M)^{1-\frac{2}{q}}=k^{\frac{2}{q}}
\end{equation*}%
we get%
\begin{equation*}
G_{q}(u)\geq k^{\frac{2}{q}}\left[ -\left( \left\Vert h\right\Vert _{\infty
}+2\left\Vert a_{+}\right\Vert _{\infty }C(\sigma )\right) +\eta k^{1-\frac{2%
}{q}}\int_{M}f^{-}dv_{g}\right]
\end{equation*}%
\begin{equation*}
\geq k^{\frac{2}{q}}\left( \left\Vert h\right\Vert _{\infty }+2\left\Vert
a_{+}\right\Vert _{\infty }C(\sigma )\right) \left( \frac{\eta k^{1-\frac{2}{%
q}}}{\left\Vert h\right\Vert _{\infty }+2\left\Vert a_{+}\right\Vert
_{\infty }C(\sigma )}\int_{M}f^{-}dv_{g}-1\right)
\end{equation*}%
and choosing $k$ such that%
\begin{equation*}
\frac{\eta k^{1-\frac{2}{q}}}{\left\Vert h\right\Vert _{\infty }+2\left\Vert
a_{+}\right\Vert _{\infty }C(\sigma )}\int_{M}f^{-}dv_{g}-1\geq 1
\end{equation*}%
that is 
\begin{equation*}
k\geq \left[ 2\frac{\left\Vert h\right\Vert _{\infty }+2\left\Vert
a_{+}\right\Vert _{\infty }C(\sigma )}{\eta \int_{M}f^{-}dv_{g}}\right] ^{%
\frac{q}{q-2}}
\end{equation*}%
we obtain%
\begin{equation*}
G_{q}(u)\geq k^{\frac{2}{q}}\left( \left\Vert h\right\Vert _{\infty
}+2\left\Vert a_{+}\right\Vert _{\infty }C(\sigma )\right) \text{.}
\end{equation*}%
Let%
\begin{equation*}
k_{1,q}=\left[ 2\frac{\left\Vert h\right\Vert _{\infty }+2\left\Vert
a_{+}\right\Vert _{\infty }C(\sigma )}{\eta \int_{M}f^{-}dv_{g}}\right] ^{%
\frac{q}{q-2}}\text{.}
\end{equation*}%
In the case $\int_{M}f^{-}\left\vert u\right\vert ^{q}dv_{g}<\eta
k\int_{M}f^{-}dv_{g}$, we have%
\begin{equation*}
\left\Vert \Delta u\right\Vert _{2}^{2}-\int_{M}a\left\vert \nabla
u\right\vert ^{2}dv_{g}\geq \lambda _{a,f,q,\eta }\left\Vert u\right\Vert
_{2}^{2}
\end{equation*}%
so 
\begin{equation*}
G_{q}(u)\geq \lambda _{a,f,\eta ,q}\left\Vert u\right\Vert
_{2}^{2}+\int_{M}hu^{2}dv_{g}+\int_{M}f^{-}\left\vert u\right\vert ^{q}dv_{g}
\end{equation*}%
\begin{equation*}
\geq (\lambda _{a},_{f,\eta ,q}-\left\Vert h\right\Vert _{\infty
})\left\Vert u\right\Vert _{2}^{2}+\int_{M}f^{-}\left\vert u\right\vert
^{q}dv_{g}
\end{equation*}%
by Lemma \ref{lem3} and \ref{lem5} there exists $\eta $ such that 
\begin{equation*}
\lambda _{a},_{f,\eta ,q}-\left\Vert h\right\Vert _{\infty }=\varepsilon
_{o}>0\text{.}
\end{equation*}%
Now, putting $\delta _{1}+\delta _{2}=\varepsilon _{o}$, where $\delta _{1}$
and $\delta _{2}$ are positive real numbers, and solving $\left\Vert
u\right\Vert _{2}^{2}$ in (\ref{13}), we get 
\begin{equation*}
\left\Vert u\right\Vert _{2}^{2}\geq \frac{1}{\left\Vert h\right\Vert
_{\infty }+2\left\Vert a_{+}\right\Vert _{\infty }C(\sigma )}\left[ \left(
1-2\sigma \left\Vert a_{+}\right\Vert _{\infty }\right) \left\Vert \Delta
u\right\Vert _{2}^{2}-G_{q}(u)+\int_{M}f^{-}\left\vert u\right\vert
^{q}dv_{g}\right] \text{.}
\end{equation*}%
Consequently%
\begin{equation*}
\left( 1+\frac{\delta _{2}}{\left\Vert h\right\Vert _{\infty }+2\left\Vert
a_{+}\right\Vert _{\infty }C(\sigma )}\right) G_{q}(u)\geq \delta
_{1}\left\Vert u\right\Vert _{2}^{2}+\frac{\delta _{2}}{\left\Vert
h\right\Vert _{\infty }+2\left\Vert a_{+}\right\Vert _{\infty }C(\sigma )}%
\left( 1-2\sigma \left\Vert a_{+}\right\Vert _{\infty }\right) \left\Vert
\Delta u\right\Vert _{2}^{2}
\end{equation*}%
so 
\begin{equation*}
G_{q}(u)\geq \frac{\delta _{1}\left( \left\Vert h\right\Vert _{\infty
}+2\left\Vert a_{+}\right\Vert _{\infty }C(\sigma )\right) }{\left\Vert
h\right\Vert _{\infty }+2\left\Vert a_{+}\right\Vert _{\infty }C(\sigma
)+\delta _{2}}\left\Vert u\right\Vert _{2}^{2}+\frac{\delta _{2}\left(
1-2\sigma \left\Vert a_{+}\right\Vert _{\infty }\right) }{\left\Vert
h\right\Vert _{\infty }+2\left\Vert a_{+}\right\Vert _{\infty }C(\sigma
)+\delta _{2}}\left\Vert \Delta u\right\Vert _{2}^{2}
\end{equation*}%
and where $\sigma $ is sufficiently small and such that $1-2\left\Vert
a_{+}\right\Vert _{\infty }\sigma >0$.

Or 
\begin{equation*}
G_{q}(u)\geq \frac{\delta _{2}\left( 1-2\sigma \left\Vert a_{+}\right\Vert
_{\infty }\right) }{\left( \left\Vert h\right\Vert _{\infty }+2\left\Vert
a_{+}\right\Vert _{\infty }C(\sigma )+\delta _{2}\right) \left(
K_{2}^{2}+\varepsilon \right) }
\end{equation*}%
\begin{equation*}
\times \left[ K_{2}^{2}\left( 1+\varepsilon \right) \left\Vert \Delta
u\right\Vert _{2}^{2}+\frac{\delta _{1}\left( \left\Vert h\right\Vert
_{\infty }+2\left\Vert a_{+}\right\Vert _{\infty }C(\sigma )\right) \left(
K_{2}^{2}+\varepsilon \right) }{\delta _{2}\left( 1-2\sigma \left\Vert
a_{+}\right\Vert _{\infty }\right) A(\varepsilon )}A(\varepsilon )\left\Vert
u\right\Vert _{2}^{2}\right]
\end{equation*}%
where for any fixed $\varepsilon >0$, $K_{2}^{2}$ denotes the best Sobolev
constant in the embedding of $H_{2}^{2}(R^{n})$ in $L^{q}(R^{n}).$

Taking $\delta _{1}$ and $\delta _{2}$ such that%
\begin{equation*}
\frac{\delta _{1}\left( \left\Vert h\right\Vert _{\infty }+2\left\Vert
a_{+}\right\Vert _{\infty }C(\sigma )\right) \left( K_{2}^{2}+\varepsilon
\right) }{\delta _{2}\left( 1-2\sigma \left\Vert a_{+}\right\Vert _{\infty
}\right) A(\varepsilon )}=1
\end{equation*}%
we get%
\begin{equation*}
\delta _{1}=\frac{\left( 1-2\sigma \left\Vert a_{+}\right\Vert _{\infty
}\right) A(\varepsilon )}{\left( \left\Vert h\right\Vert _{\infty
}+2\left\Vert a_{+}\right\Vert _{\infty }C(\sigma )\right) K_{2}^{2}\left(
1+\varepsilon \right) +\left( 1-2\sigma \left\Vert a_{+}\right\Vert _{\infty
}\right) A(\varepsilon )}\varepsilon _{o}
\end{equation*}%
and%
\begin{equation*}
\delta _{2}=\frac{\left( \left\Vert h\right\Vert _{\infty }+2\left\Vert
a_{+}\right\Vert _{\infty }C(\sigma )\right) \left( K_{2}^{2}+\varepsilon
\right) }{\left( \left\Vert h\right\Vert _{\infty }+2\left\Vert
a_{+}\right\Vert _{\infty }C(\sigma )\right) K_{2}^{2}\left( 1+\varepsilon
\right) +\left( 1-2\sigma \left\Vert a_{+}\right\Vert _{\infty }\right)
A(\varepsilon )}\varepsilon _{o}.
\end{equation*}%
Consequently%
\begin{equation*}
G_{q}(u)\geq \frac{\delta _{2}\left( 1-2\sigma \left\Vert a_{+}\right\Vert
_{\infty }\right) }{\left( \left\Vert h\right\Vert _{\infty }+2\left\Vert
a_{+}\right\Vert _{\infty }C(\sigma )+\delta _{2}\right) \left(
K_{2}^{2}+\varepsilon \right) }\left\Vert u\right\Vert _{q}^{2}
\end{equation*}%
and since%
\begin{equation*}
\left\Vert h\right\Vert _{\infty }+2\left\Vert a_{+}\right\Vert _{\infty
}C(\sigma )+\delta _{2}=\left( \left\Vert h\right\Vert _{\infty
}+2\left\Vert a_{+}\right\Vert _{\infty }C(\sigma )\right)
\end{equation*}%
\begin{equation*}
\times \left[ 1+\frac{K_{2}^{2}\left( 1+\varepsilon \right) }{\left(
\left\Vert h\right\Vert _{\infty }+2\left\Vert a_{+}\right\Vert _{\infty
}C(\sigma )\right) K_{2}^{2}\left( 1+\varepsilon \right) +\left( 1-2\sigma
\left\Vert a_{+}\right\Vert _{\infty }\right) A(\varepsilon )}\varepsilon
_{o}\right]
\end{equation*}%
\begin{equation*}
=\frac{\left( \varepsilon _{o}+\left\Vert h\right\Vert _{\infty
}+2\left\Vert a_{+}\right\Vert _{\infty }C(\sigma )\right) K_{2}^{2}\left(
1+\varepsilon \right) +\left( 1-2\sigma \left\Vert a_{+}\right\Vert _{\infty
}\right) A(\varepsilon )}{\left( \left\Vert h\right\Vert _{\infty
}+2\left\Vert a_{+}\right\Vert _{\infty }C(\sigma )\right) K_{2}^{2}\left(
1+\varepsilon \right) +\left( 1-2\sigma \left\Vert a_{+}\right\Vert _{\infty
}\right) A(\varepsilon )}
\end{equation*}%
we get that%
\begin{equation*}
G_{q}(u)\geq \frac{\left( 1-2\sigma \left\Vert a_{+}\right\Vert _{\infty
}\right) \varepsilon _{o}}{\left[ \left( \varepsilon _{o}+\left\Vert
h\right\Vert _{\infty }+2\left\Vert a_{+}\right\Vert _{\infty }C(\sigma
)\right) K_{2}^{2}\left( 1+\varepsilon \right) +\left( 1-2\sigma \left\Vert
a_{+}\right\Vert _{\infty }\right) A(\varepsilon )\right] }k^{\frac{2}{q}}%
\text{.}
\end{equation*}%
Letting 
\begin{equation*}
b=\frac{\left( 1-2\sigma \left\Vert a_{+}\right\Vert _{\infty }\right)
\varepsilon _{o}}{\left[ \left( \varepsilon _{o}+\left\Vert h\right\Vert
_{\infty }+2\left\Vert a_{+}\right\Vert _{\infty }C(\sigma )\right)
K_{2}^{2}\left( 1+\varepsilon \right) +\left( 1-2\sigma \left\Vert
a_{+}\right\Vert _{\infty }\right) A(\varepsilon )\right] }
\end{equation*}%
we get 
\begin{equation*}
F_{q}(u)=G_{q}(u)-\int_{M}f^{+}\left\vert u\right\vert ^{q}dv_{g}
\end{equation*}%
\begin{equation*}
\geq bk^{\frac{2}{q}}-\int_{M}f^{+}\left\vert u\right\vert ^{q}dv_{g}\geq
bk^{\frac{2}{q}}-k\sup f^{+}=k^{\frac{2}{q}}(b-k^{1-\frac{^{2}}{q}}\sup
f^{+}).
\end{equation*}%
So if $\sup_{M}f>0$, let $\mu =\inf \left( b,\left\Vert h\right\Vert
_{\infty }+2\left\Vert a_{+}\right\Vert _{\infty }C(\sigma )\right) $. For
any $k\geq k_{1,q}$, we have%
\begin{equation*}
F_{q}(u)\geq k^{\frac{2}{q}}(\mu -k^{1-\frac{^{2}}{q}}\sup f)
\end{equation*}%
Now if we put $C_{q}=\frac{\eta }{8\left( \left\Vert h\right\Vert _{\infty
}+2\left\Vert a_{+}\right\Vert _{\infty }C(\sigma )\right) }\mu $ and
suppose that $\sup_{M}f\leq C_{q}\int_{M}f^{-}$, we obtain that the
inequality is fulfilled provided that%
\begin{equation*}
k\leq \left[ \frac{4\left( \left\Vert h\right\Vert _{\infty }+2\left\Vert
a_{+}\right\Vert _{\infty }C(\sigma )\right) }{\eta \int_{M}f^{-}dv_{g}}%
\right] ^{\frac{q}{q-2}}=2^{\frac{q}{q-2}}k_{1,q}\text{.}
\end{equation*}%
and%
\begin{equation*}
F_{q}(u)\geq \frac{1}{2}\mu k^{\frac{2}{q}}
\end{equation*}%
provided that 
\begin{equation*}
k\leq \left[ \frac{\mu }{2\sup f}\right] ^{\frac{q}{q-2}}.
\end{equation*}%
We put%
\begin{equation*}
k_{2,q}=2^{\frac{q}{q-2}}k_{1,q}.
\end{equation*}

\textit{Case }$f^{+}=0$.

In this case, for any $k\geq k_{1,q}$, $\ $%
\begin{equation*}
F_{q}(u)\geq \frac{1}{2}\mu k^{\frac{2}{q}}.
\end{equation*}
\end{proof}

\section{{}Subcritical case}

First, we show the existence of a solution to the subcritical equation with
negative energy.

\begin{lemma}
\label{lem7} For each $t>0,$ small enough, $\inf\limits_{\left\Vert
u\right\Vert _{H_{2}}\leq t}F_{q}(u)<0,$ $\ \ \ q\in \left] 2,N\right] $.
\end{lemma}

In fact $F_{q}(t)\leq t^{2}\left( h-t^{q-2}\int_{M}fdv_{g}\right) ,$ where $%
h=\max_{M}h(x)$, and since $h<0$, there is $t_{o}>0$ small enough such that $%
\inf\limits_{\left\Vert u\right\Vert _{H_{2}}\leq t}F_{q}(u)<0$ for each $%
t\in \left] 0,t_{o}\right[ .$

\begin{proposition}
\label{prop3} Let $a$, $h$ be $C^{\infty }$ functions on $M$ , with $h$
negative.\ For every $C^{\infty }$ function, $f$ on $M$ with $%
\int_{M}f^{-}dv_{g}>0$, \ there exists a constant $C>0$ which depends only
on $\frac{f^{-}}{\int f^{-}dv_{g}}$ such that if $f$ \ satisfies the
following conditions 
\begin{equation*}
\left. 
\begin{array}{l}
(1)\text{ }\left\vert h(x)\right\vert <\lambda _{a,f}\text{ \ \ \ \ \ \ \
for any }x\in M\text{\ \ \ \ \ \ \ \ \ \ \ \ \ \ \ \ \ \ \ \ \ \ \ \ \ \ \ \
\ \ \ \ \ \ \ \ \ \ \ \ \ \ \ \ \ \ \ \ \ \ \ \ \ \ \ \ \ \ \ \ \ \ \ \ \ \
\ \ \ \ \ \ \ \ \ \ \ \ \ \ \ \ \ \ \ \ \ \ \ \ \ } \\ 
(2)\text{ }\frac{\sup f^{+}}{\int f^{-}dv_{g}}<C%
\end{array}%
\right.
\end{equation*}%
then the subcritical equation%
\begin{equation}
\Delta ^{2}u_{q}+\nabla ^{i}(a\nabla _{i}u_{q})+hu_{q}=f\left\vert
u_{q}\right\vert ^{q-2}u\text{ }_{q}\text{\ \ \ with }q\in \left] 2,N\right[
\label{14}
\end{equation}%
admits a $C^{4,\alpha }$, for some $\alpha \in (0,1)$, solution $u_{q}$\
with negative energy.
\end{proposition}

\begin{proof}
For any $q\in \left] 2,N\right[ $ and $k>0$, let $\mu
_{k,q}=\inf_{\left\Vert w\right\Vert _{q}^{q}=k}F_{q}(w)$. First we remark
that if $k$ is close to $0$, $k>0$, $\mu _{k,q}<0\ $: indeed \ 
\begin{equation*}
\mu _{k,q}\leq F_{q}(k^{\frac{1}{q}})=k^{\frac{2}{q}}\left(
\int_{M}hdv_{g}-k^{1-\frac{2}{q}}\int_{M}fdv_{g}\right) <0\text{.}
\end{equation*}%
By Proposition \ref{prop2} the real valued function $k\rightarrow \mu _{k,q}$
is continuous and $\mu _{k,q}$ goes to $0$ , when $k\rightarrow 0$. So by
Lemma \ref{lem6} and \ref{lem7} the function $k\rightarrow \mu _{k,q}$
starts at $0$, takes a negative minimum, say at $k_{q}$, then takes positive
values. Let $l_{q}=k_{1,q}=\left[ 2\frac{\left\Vert h\right\Vert _{\infty
}+2\left\Vert a_{+}\right\Vert _{\infty }C(\sigma )}{\eta \int_{M}f^{-}dv_{g}%
}\right] ^{\frac{q}{q-2}}$ the lower bound of the interval $I_{q}$ given in
the proof of Lemma \ref{lem6}, then 
\begin{equation*}
\mu _{k_{q},q}=\inf_{\left\Vert u\right\Vert _{q}^{q}\leq l_{q}}F_{q}(u)%
\text{.}
\end{equation*}%
By Proposition \ref{prop1} the infimum $\mu _{kq,q}$ is attained by a
function $v_{q}\in H_{2}$ with $\left\Vert v_{q}\right\Vert _{q}^{q}=k_{q}$
, so%
\begin{equation*}
F_{q}(v_{q})=\inf_{\left\Vert u\right\Vert _{q}^{q}\leq l_{q}}F_{q}(u)\ 
\text{.}
\end{equation*}%
Now since for any $k_{q}\in I_{q}$, and any $u\in H_{2}$ with $\left\Vert
u\right\Vert _{q}^{q}=k_{q}$, $F_{q}(u)\geq 0$, it follows that $k_{q}<l_{q}$%
. So $v_{q}$ is a critical point of $F_{q}$ , that is for any $\varphi \in
H_{2}$ 
\begin{equation*}
\int_{M}\Delta v_{q}\Delta \varphi dv_{g}-\int_{M}a\nabla v_{q}\nabla
\varphi dv_{g}+
\end{equation*}%
\begin{equation*}
\int_{M}hv_{q}\varphi dv_{g}-\frac{q}{2}\int_{M}f\left\vert v_{q}\right\vert
^{q-2}v_{q}\varphi dv_{g}=0
\end{equation*}%
then $u_{q}=(\frac{q}{2})^{\frac{1}{q-2}}v_{q}$ is a weak solution of the
subcritical equation with negative energy such that 
\begin{equation*}
\left\Vert u_{q}\right\Vert _{q}^{q}\leq (\frac{q}{2})^{\frac{q}{q-2}}l_{q}%
\text{.}
\end{equation*}%
Moreover, arguing as in the proof of the Proposition \ref{prop1}, $u_{q}\in
C^{4,\alpha }(M)$ \ with $\alpha \in (0,1)$.
\end{proof}

Now we are going to seek a second solution to the subcritical equation with
positive energy.

We start by showing that $F_{q}$ with $q\in \left] 2,N\right[ $ satisfies
the Palais-Smale condition.

\begin{lemma}
\label{lem8} Let $c\ $be a real number, then each Palais-Smale sequence at
level $c$ for the functional $F_{q}$ satisfies the Palais -Smale condition.
\end{lemma}

\begin{proof}
First, we show that each Palais-Smale sequence is bounded: we argue by
contradiction. Suppose that there exists a sequence $\left( u_{j}\right) $
such that $F_{q}(u_{j})$ tends to a finite limit $c$, $F_{q}^{\prime
}(u_{j}) $ goes strongly to zero and $u_{j}$ to infinite in the $H_{2}$%
-norm. More explicitly we have 
\begin{equation*}
\int_{M}\left( (\Delta u_{j})^{2}-a\left\vert \nabla u_{j}\right\vert
^{2}+hu_{j}^{2}\right) dv_{g}-\int_{M}f\left\vert u\right\vert
_{j}^{q}dv_{g}\rightarrow c
\end{equation*}%
and 
\begin{equation*}
\int_{M}\left( \left( \Delta u_{j}\right) ^{2}-a\left\vert \nabla
u_{j}\right\vert ^{2}+hu_{j}^{2}\right) dv_{g}-\frac{q}{2}%
\int_{M}f\left\vert u\right\vert _{j}^{q-1}vdv_{g}\rightarrow 0
\end{equation*}%
so for any $\varepsilon >0$ there exists a positive integer $A$ such that
for every $j\geq A$ we have 
\begin{equation*}
\left\vert \int_{M}\left( (\Delta u_{j})^{2}-a\left\vert \nabla
u_{j}\right\vert ^{2}+hu_{j}^{2}\right) dv_{g}-\int_{M}f\left\vert
u\right\vert _{j}^{q}dv_{g}-c\right\vert \leq \varepsilon
\end{equation*}%
and%
\begin{equation*}
\left\vert \int_{M}\left( \left( \Delta u_{j}\right) ^{2}-a\left\vert \nabla
u_{j}\right\vert ^{2}+hu_{j}^{2}\right) dv_{g}dv_{g}-\frac{q}{2}%
\int_{M}f\left\vert u\right\vert _{j}^{q-1}vdv_{g}\right\vert \leq
\varepsilon \text{.}
\end{equation*}

Hence, we get 
\begin{equation}
\left\vert (q-2)\int_{M}(\Delta u_{j})^{2}-a\left\vert \nabla
u_{j}\right\vert ^{2}+hu_{j}^{2}dv_{g}-qc\right\vert \leq (q+2)\epsilon
\label{18}
\end{equation}%
and%
\begin{equation}
\left\vert (q-2)\int_{M}f\left\vert u_{j}\right\vert ^{q}-2c\right\vert \leq
4\varepsilon .  \label{19}
\end{equation}%
By Lemma \ref{lem6}, we can choose $k$ to be an $L^{q}-$ norm such that 
\begin{equation*}
\inf_{\left\Vert u\right\Vert _{q}^{q}=k}F_{q}(u)>0\text{.}
\end{equation*}%
Letting $v_{j}=k^{\frac{1}{q}}\frac{u_{j}}{\left\Vert u_{j}\right\Vert _{q}}%
, $ we obtain from (\ref{18}) and (\ref{19}) that 
\begin{equation}
\left\vert (q-2)\int_{M}f\left\vert v_{j}\right\vert ^{q}dv_{g}-\frac{2ck^{%
\frac{2}{q}}}{\left\Vert u_{j}\right\Vert _{q}^{2}}\right\vert \leq
4\varepsilon \frac{k^{\frac{2}{q}}}{\left\Vert u_{j}\right\Vert _{q}^{2}}
\label{20}
\end{equation}%
and

\begin{equation*}
\left\vert (q-2)\int_{M}(\Delta v_{j})^{2}-a\left\vert \nabla
v_{j}\right\vert ^{2}+hv_{j}^{2}dv_{g}-qc\frac{k^{\frac{2}{q}}}{\left\Vert
u_{j}\right\Vert _{q}^{2}}\right\vert
\end{equation*}%
\begin{equation}
\leq (q+2)\epsilon \frac{k^{\frac{2}{q}}}{\left\Vert u_{j}\right\Vert
_{q}^{2}}\text{.}  \label{21}
\end{equation}

Now since ($\left\Vert v_{j}\right\Vert _{q}$)$_{j}$ is a bounded sequence,
it follows by (\ref{21}) that $\left( v_{j}\right) $ is bounded in $H_{2}$.
If $\left\Vert u_{j}\right\Vert _{q}$ goes to infinity, it follows from (\ref%
{20}) and (\ref{21}) that $F_{q}(v_{j})$ goes to zero. And since $\left\Vert
v_{j}\right\Vert _{q}^{q}=k,$ we have 
\begin{equation*}
\inf_{\left\Vert u\right\Vert _{q}^{q}=k}F_{q}(u)\leq F_{q}(v_{j})
\end{equation*}%
so%
\begin{equation*}
\inf_{\left\Vert u\right\Vert _{q}^{q}=k}F_{q}(u)\leq 0\text{.}
\end{equation*}%
Hence a contradiction. Then the sequence $\left( u_{j}\right) $ is bounded
in $H_{2}$. Since $q<N$, the Sobolev injections are compact. Consequently
the Palais-Smale condition is satisfied.
\end{proof}

\begin{lemma}
\label{lem9} Let $u\in H_{2}$. If the $L_{q}$-norm $\left\Vert u\right\Vert
_{q}^{q}=k$ goes to $+\infty $, then \ $\mu _{k,q}=\inf_{\left\Vert
u\right\Vert _{q}^{q}=k}F_{q}(u)\rightarrow -\infty $ .
\end{lemma}

\begin{proof}
In fact since $\sup_{x\in M}f(x)>0$ let $u$ be a function of class $C^{2}$
with support contained in the open subset $\left\{ x\in M:f(x)>0\right\} $
of the manifold $M$ such that $\left\Vert u\right\Vert _{q}^{q}=1$, then $%
\int_{M}f\left\vert u\right\vert ^{q}dv_{g}>0$ and 
\begin{equation*}
F_{q}(ku)=k^{\frac{2}{q}}\left( \int_{M}\left( (\Delta u)^{2}-a\left\vert
\nabla u\right\vert ^{2}+hu^{2}\right) dv_{g}-k^{\frac{q-2}{q}%
}\int_{M}f\left\vert u\right\vert ^{q}dv_{g}\right) \text{.}
\end{equation*}%
So $\lim_{k\rightarrow +\infty }F_{q}(ku)=-\infty $ .
\end{proof}

\begin{proposition}
\label{prop4} Let $a$, $h$ be $C^{\infty }$ functions on $M$ with $h$
negative.\ For every $C^{\infty }$ function, $f$ on $M$ with $%
\int_{M}f^{-}>0 $, there exists a constant $C>0$ which depends only on $%
\frac{f^{-}}{\int f^{-}}$ such that if $f$ satisfies the following
conditions 
\begin{equation*}
\left. 
\begin{array}{l}
(1)\text{ }\left\vert h(x)\right\vert <\lambda _{a,f}\text{ \ \ \ \ \ \ \ \
\ \ \ \ \ \ \ \ \ for any }x\in M\text{\ \ \ \ \ \ \ \ \ \ \ \ \ \ \ \ \ \ \
\ \ \ \ \ \ \ \ \ \ \ \ \ \ \ \ \ \ \ \ \ \ \ \ \ \ \ \ \ \ \ \ \ \ \ \ \ \
\ \ \ \ \ \ \ \ \ \ \ \ \ \ \ \ } \\ 
(2)\text{ }\frac{\sup f^{+}}{\int f^{-}}<C \\ 
(3)\text{ }\sup f>0,%
\end{array}%
\right.
\end{equation*}%
then the subcritical equation 
\begin{equation*}
\Delta ^{2}u+\nabla ^{i}(a\nabla _{i}u)+hu=f\left\vert u\right\vert ^{q-2}u,%
\text{ \ \ \ \ }q\in \left] 2,N\right[
\end{equation*}%
admits a nontrivial solution of class $C^{4,\alpha }$, for some $\alpha \in
\left( 0,1\right) $, with positive energy.
\end{proposition}

\begin{proof}
By Lemma \ref{lem6}, \ref{lem7} and \ref{lem9} the curve $k\rightarrow \mu
_{k,q}$ starts at $0$, takes a negative minimum, then takes positive \
maximum and goes to minus infinite. Mimicking which is done in (\cite{11}),
let $l_{o}$ be an $L^{q}$-norm such that $\mu _{l_{o},q}$ is a maximum and $%
l_{1}$, $l_{2}$ two $L^{q}$-norms such that $\mu _{l_{1},q}=\mu _{l_{2},q}=0$
with $l_{1}<l_{o}$ and $l_{2}>l_{o}$.

Set 
\begin{equation*}
\Gamma =\left\{ \gamma \in C\left( \left[ 0,1\right] ,H_{2}\right) \text{: }%
\gamma (0)=u_{l_{1},q}\text{,}\gamma (1)=u_{l_{2},q}\right\} \text{,}
\end{equation*}%
where $u_{l_{i},q}\in B_{l_{i},q}$, $i=1,2$, are such that $\mu
_{l_{i},q}=F_{q}\left( u_{l_{i},q}\right) =\inf_{w\in
B_{l_{i},q}}F_{q}\left( w\right) $

and 
\begin{equation*}
\nu _{q}=\inf_{\gamma \in \Gamma }\max_{t\in \left[ 0,1\right] }F_{q}\left(
\gamma (t)\right) \text{ .}
\end{equation*}%
Arguing as in \cite{11}, we show that $\nu _{q}$ is a critical level of the
functional $F_{q}$ and $\nu _{q}\geq \mu _{l,q}>0$. Consequently the
subcritical equation (\ref{2}) admits a weak solution of positive energy.
This solution is in fact of class $C^{4,\alpha }$ with $\alpha \in (0,1)$.
\end{proof}

Theorem \ref{th1} follows from Proposition \ref{prop3} and \ref{prop4}.

\section{Critical case}

Now, we are going to investigate solutions of the critical equation.

\begin{theorem}
\label{th5} Let $a$, $h$ be $C^{\infty }$ functions on $M$ with $h$
negative.\ For every $C^{\infty }$ function, $f$ on $M$ with $%
\int_{M}f^{-}>0 $, \ there exists a constant $C>0$ which depends only on $%
\frac{f^{-}}{\int f^{-}}$ such that if $f$ satisfies the following
conditions 
\begin{equation*}
\left. 
\begin{array}{l}
(1)\text{ }\left\vert h(x)\right\vert <\lambda _{a},_{f}\text{ \ \ \ \ \ \ \
for any }x\in M\text{\ \ \ \ \ \ \ \ \ \ \ \ \ \ \ \ \ \ \ \ \ \ \ \ \ \ \ \
\ \ \ \ \ \ \ \ \ \ \ \ \ \ \ \ \ \ \ \ \ \ \ \ \ \ \ \ \ \ \ \ \ \ \ \ \ \
\ \ \ \ \ \ \ \ \ \ \ \ \ \ \ \ \ \ \ \ \ \ \ \ \ } \\ 
(2)\text{ }\frac{\sup f^{+}}{\int f^{-}}<C%
\end{array}%
\right.
\end{equation*}%
then the critical equation%
\begin{equation}
\Delta ^{2}u+\nabla ^{i}(a\nabla _{i}u)+hu=f\left\vert u\right\vert ^{N-2}u%
\text{ \ \ }  \label{e}
\end{equation}%
admits a $C^{4,\alpha }$, for some $\alpha \in (0,1)$, solution $u$\ with
negative energy.
\end{theorem}

\begin{proof}
For each $q\in \left( 2,N\right) $, let $u_{q}$ be the solution to the
subcritical equation (\ref{14}) given by Proposition \ref{prop3}, $u_{q}$ is
of negative energy. We have already shown in the proof of Proposition \ref%
{prop3} that 
\begin{equation*}
\left\Vert u_{q}\right\Vert _{q}^{q}=k_{q}\leq l_{q}=\left[ 2\frac{%
\left\Vert h\right\Vert _{\infty }+2\left\Vert a_{+}\right\Vert _{\infty
}C(\sigma )}{\eta \int_{M}f^{-}dv_{g}}\right] ^{\frac{q}{q-2}}
\end{equation*}%
and since $l_{q}$ goes to $l_{N}=\left[ 2\frac{\left\Vert h\right\Vert
_{\infty }+2\left\Vert a_{+}\right\Vert _{\infty }C(\sigma )}{\eta
\int_{M}f^{-}dv_{g}}\right] ^{\frac{4}{n}}$ as $q$ goes to $N$, $(u_{q})$ is
bounded in $L^{q}$, so it is in $L^{2}$ and since $u_{q}$ are of negative
energy then%
\begin{equation*}
\left\Vert \Delta u_{q}\right\Vert _{2}^{2}\leq \int_{M}a\left\vert \nabla
u\right\vert ^{2}dv_{g}-\int_{M}hu_{q}^{2}dv_{g}+\int_{M}f\left\vert
u_{q}\right\vert ^{q}dv_{g}
\end{equation*}%
\begin{equation*}
\leq \left\Vert a_{+}\right\Vert _{\infty }\left\Vert \nabla
u_{q}\right\Vert _{2}^{2}+\left\Vert h\right\Vert _{\infty }\left\Vert
u_{q}\right\Vert _{q}^{2}+\left\Vert f\right\Vert _{\infty }\left\Vert
u_{q}\right\Vert _{q}^{q}\text{.}
\end{equation*}

Now since for any sufficiently $\sigma >0$, there exists a constant $%
C(\sigma )$ such that%
\begin{equation*}
\left\Vert \nabla u_{q}\right\Vert _{2}^{2}\leq 2\sigma \left\Vert \Delta
u_{q}\right\Vert _{2}^{2}+2C(\sigma )\left\Vert u_{q}\right\Vert _{2}^{2}
\end{equation*}%
we get 
\begin{equation*}
\left( 1-2\sigma \left\Vert a_{+}\right\Vert _{\infty }\right) \left\Vert
\Delta u_{q}\right\Vert _{2}^{2}\leq \left( 2\left\Vert a_{+}\right\Vert
_{\infty }C(\sigma )+\left\Vert h\right\Vert _{\infty }\right) \left\Vert
u_{q}\right\Vert _{q}^{2}+\left\Vert f\right\Vert _{\infty }\left\Vert
u_{q}\right\Vert _{q}^{q}
\end{equation*}%
\begin{equation*}
\leq \left( 2\left\Vert a_{+}\right\Vert _{\infty }C(\sigma )+\left\Vert
h\right\Vert _{\infty }\right) l_{q}^{\frac{2}{q}}+\left\Vert f\right\Vert
_{\infty }l_{q}\text{.}
\end{equation*}%
So $\left( u_{q}\right) $ is a bounded sequence in $H_{2}.$ Consequently $%
u_{q}\rightarrow v$ weakly in $H_{2}$, up to a subsequence, we have

\begin{equation*}
u_{q}\rightarrow v\text{ \ strongly in }L^{s}(M)\text{ \ \ \ for }s<N
\end{equation*}%
\begin{equation*}
\nabla u_{q}\rightarrow \nabla v\text{ \ strongly in\ }L^{2}\ 
\end{equation*}%
\begin{equation*}
u_{q}(x)\rightarrow v(x)\ \ \text{ \ \ \ for \ a.e. }x\in M.
\end{equation*}%
On the other hand for any $q\in \left] 2,N\right[ $, $u_{q}$ satisfies, for
any $\varphi \in H_{2}$%
\begin{equation*}
\int_{M}\Delta u_{q}\Delta \varphi dv_{g}-\int_{M}a\nabla ^{i}u_{q}\nabla
_{i}\varphi dv_{g}+\int_{M}hu_{q}\varphi dv_{g}\text{ }
\end{equation*}%
\begin{equation}
=\frac{q}{2}\int_{M}f\left\vert u_{q}\right\vert ^{q-2}u_{q}\text{ }\varphi
dv_{g}  \label{15'}
\end{equation}%
and since the convergence of $(u_{q})$ is weak in $H_{2}$,\ it follows that
for any $\varphi \in H_{2}$%
\begin{equation*}
\int_{M}\Delta u_{q}\Delta \varphi dv_{g}-\int_{M}a\nabla ^{i}u_{q}\nabla
_{i}\varphi dv_{g}+\int_{M}hu_{q}\varphi dv_{g}
\end{equation*}%
\begin{equation}
\rightarrow \int_{M}\Delta v\Delta \varphi dv_{g}-\int_{M}a\nabla
^{i}v\nabla _{i}\varphi )dv_{g}+\int_{M}hv\varphi dv_{g}\text{ .}  \label{16}
\end{equation}%
Moreover since $u_{q}(x)$ $\rightarrow v(x)$ for a.e. $x\in M$ and $(u_{q})$
is bounded in $H_{2}$ we have%
\begin{equation*}
u_{q}(x)\left\vert u_{q}(x)\right\vert ^{q-2}\rightarrow v(x)\left\vert
v(x)\right\vert ^{N-2}\text{ \ \ \ for a.e. }x\in M
\end{equation*}%
and%
\begin{equation*}
\left\Vert u_{q}\left\vert u_{q}\right\vert ^{q-2}\right\Vert _{\frac{N}{N-1}%
}=\left\Vert u_{q}\right\Vert _{(q-1)\frac{N}{N-1}}^{q-1}\leq
C_{1}\left\Vert u_{q}\right\Vert _{N}^{N-1}\leq C\left\Vert u_{q}\right\Vert
_{H_{2}}^{N-1}\text{.}
\end{equation*}%
consequently $(u_{q})$ is bounded in $L^{\frac{N}{N-1}}$ and by a well known
theorem \cite{1} $u_{q}$ converges weakly to $v$ in $L^{\frac{N}{N-1}}$. Now
for any $\varphi \in H_{2}$ $\subset L^{N},$ and any smooth function $f,$ $%
f\varphi \in L^{N}$ ( the dual space of $L^{\frac{N}{N-1}}$), then 
\begin{equation}
\int_{M}f\left\vert u_{q}\right\vert ^{q-2}u_{q}\varphi dv_{g}\rightarrow
\int_{M}f\left\vert v\right\vert ^{N-2}v\varphi dv_{g}\text{.}  \label{17}
\end{equation}%
So by (\ref{16}) and (\ref{17}) $u=$ $\left( \frac{N}{2}\right) ^{\frac{1}{%
N-2}}v$ is a weak solution of the critical equation.

It remains to check that $u\neq 0$. We let 
\begin{equation*}
\mu _{k_{q},q}=\inf_{w\in \overset{\_}{B}_{k,q}}F_{q}(w)
\end{equation*}%
where 
\begin{equation*}
\overset{\_}{B}_{k,q}=\left\{ w\in H_{2}\left( M\right) :\left\Vert
w\right\Vert _{q}^{q}\leq l_{q}\right\} \text{.}
\end{equation*}%
By Proposition \ref{prop1}, $\mu _{k_{q},q}$ is attained by by a function $%
u_{q}\in H_{2}\left( M\right) $ with $\left\Vert u_{q}\right\Vert =k_{q}\leq
l_{q}$ that is $\mu _{k_{q},q}=F_{q}(u_{q})$.

\begin{claim}
$\mu _{k_{q},q}$ are uniformly lower bounded, as $q$ goes to $N$.
\end{claim}

Indeed, in one hand we have $\mu _{k_{q},q}<0$ and on the other hand if $%
\min_{x\in M}a(x)\leq 0$ we obtain%
\begin{equation*}
\mu _{k_{q},q}=F_{q}(u_{q})
\end{equation*}%
\begin{equation*}
=\left\Vert \Delta u_{q}\right\Vert _{2}^{2}-\int_{M}a\left\vert \nabla
u_{q}\right\vert ^{2}dv_{g}+\int_{M}hu_{q}^{2}dv_{g}-\int_{M}f\left\vert
u_{q}\right\vert ^{q}dv_{g}
\end{equation*}%
\begin{equation*}
\geq \min_{x\in M}h(x)k_{q}^{\frac{2}{q}}-\max_{x\in M}f^{+}(x)k_{q}\text{.}
\end{equation*}%
Letting 
\begin{equation*}
C_{q}=\max (l_{q},1)
\end{equation*}%
we get%
\begin{equation*}
\mu _{k_{q},q}\geq \left( \min_{x\in M}h(x)-\max_{x\in M}f^{+}(x)\right)
C_{q}
\end{equation*}%
so%
\begin{equation*}
\lim_{q\rightarrow N}\inf \mu _{k_{q},q}\geq \left( \min_{x\in
M}h(x)-\max_{x\in M}f^{+}(x)\right) C_{N}\text{.}
\end{equation*}%
In the case $\min_{x\in M}a(x)>0$, thanks to formula (\ref{9}), we obtain
for any sufficiently small $\sigma >0$ 
\begin{equation*}
\mu _{k_{q},q}\geq (1-\sigma \min_{x\in M}a(x))\left\Vert \Delta
u_{q}\right\Vert _{2}^{2}+\left( \min_{x\in M}h(x)+\min_{x\in M}a(x)C(\sigma
)-\max_{x\in M}f^{+}(x)\right) C_{q}
\end{equation*}%
and taking $\sigma $ small so that $(1-\sigma \min_{x\in M}a(x))\geq 0$, we
obtain%
\begin{equation*}
\mu _{k_{q},q}\geq \left( \min_{x\in M}h(x)+\min_{x\in M}a(x)C(\sigma
)-\max_{x\in M}f^{+}(x)\right) C_{q}
\end{equation*}%
and $\mu _{k_{q},q}$ are lower bounded as $q\rightarrow N$.

\begin{claim}
Up to a subsequence we have%
\begin{equation*}
\lim_{q\rightarrow N}\mu _{k_{q},q}=\mu _{k_{N},N}<0\text{ .}
\end{equation*}
\end{claim}

For $q$ close to $N$, we let 
\begin{equation*}
0<k<\min \left( l_{q},\left[ \frac{\left\vert \int_{M}hdv_{g}\right\vert }{%
2\int_{M}f^{-}dv_{g}}\right] ^{\frac{q}{q-2}}\right) \text{.}
\end{equation*}%
Since 
\begin{equation*}
\mu _{k_{q},q}=\inf_{u\in \overset{\_}{B}_{k,q}}F_{q}(u)
\end{equation*}%
with 
\begin{equation*}
\overset{\_}{B}_{k_{q},q}=\left\{ u\in H_{2}:\left\Vert u\right\Vert
_{q}^{q}\leq l_{q}\right\}
\end{equation*}%
we get%
\begin{equation*}
\mu _{k_{q},q}\leq F_{q}(k^{\frac{1}{q}})=k^{\frac{2}{q}}\left(
\int_{M}hdv_{g}+k^{1-\frac{2}{q}}\int_{M}f^{-}dv_{g}\right)
\end{equation*}%
\begin{equation*}
\leq \frac{1}{2}k^{\frac{2}{q}}\int_{M}hdv_{g}
\end{equation*}%
hence up to a subsequence 
\begin{equation}
\mu _{k_{N},N}=\lim_{q\rightarrow N}\mu _{k_{q},q}\leq \frac{1}{2}k^{\frac{2%
}{N}}\int_{M}hdv_{g}<0\text{.}  \label{17'}
\end{equation}

Now, we are in position to show that $u=$ $\left( \frac{N}{2}\right) ^{\frac{%
1}{N-2}}v$ $\neq 0$.$\ $

\begin{claim}
The weak solution of the critical equation $\left( \text{\ref{e}}\right) $
is non trivial.
\end{claim}

In fact since $u$ is a solution of the equation (\ref{e}) and the sequence $%
\left( u_{q}\right) $, of solutions to the subcritical equations, converges
weakly to $v$ in $H_{2}$, we have%
\begin{equation}
\frac{N}{2}\int_{M}f\left\vert v\right\vert ^{N}=\left( \left\Vert \Delta
v\right\Vert _{2}^{2}-\int_{M}a\left\vert \nabla v\right\vert
^{2}dv_{g}+\int_{M}hv^{2}dv_{g}\right)  \label{17''}
\end{equation}%
\begin{equation*}
\leq \lim \inf_{q\rightarrow N}\left( \left\Vert \Delta u_{q}\right\Vert
_{2}^{2}-\int_{M}a\left\vert \nabla u_{q}\right\vert
^{2}dv_{g}+\int_{M}hu_{q}^{2}dv_{g}\right)
\end{equation*}%
\begin{equation*}
=\lim \inf_{q\rightarrow N}\left( \frac{2}{q}\int_{M}f\left\vert
u_{q}\right\vert ^{q}dv_{g}\right) .
\end{equation*}%
The function $u_{q}$ solution of the subcritical equation achieves the
minimum $\mu _{k_{q},q}=\inf_{u\in \overset{\_}{B}_{k},q}F_{q}(u)$ , where $%
\overset{\_}{B}_{k_{q},q}=\left\{ u\in H_{2}:\left\Vert u\right\Vert
_{q}^{q}\leq l_{q}\right\} $.

So%
\begin{equation*}
\mu _{k_{q},q}=F_{q}(u_{q})=\left( \frac{q}{2}-1\right) \int_{M}f\left\vert
u_{q}\right\vert ^{q}dv_{g}
\end{equation*}%
and taking account of (\ref{17'}) and (\ref{17''}), we get 
\begin{equation*}
\int_{M}f\left\vert v\right\vert ^{N}dv_{g}<0
\end{equation*}%
hence 
\begin{equation*}
u=\left( \frac{N}{2}\right) ^{\frac{1}{N-2}}v\neq 0\text{.}
\end{equation*}%
By the bootstrap method and a method imagined by Vaugon see \cite{12}, we
get that $u$ is of class $C^{4,\alpha }$ for some $\alpha \in (0,1)$.
\end{proof}

\end{document}